\documentclass[a4paper,12pt]{article}
\usepackage{amsmath}
\usepackage{bm}
\usepackage{amssymb}
\usepackage{graphicx}
\usepackage{setspace}
\usepackage{enumitem}
\usepackage{lscape}

\usepackage{amsthm}
\newtheorem{thm}{Theorem}[section]  %[chapter]
\newtheorem{lemma}[thm]{Lemma}

\newtheorem{prp}[thm]{Proposition}
\theoremstyle{definition}
\newtheorem{dfn}[thm]{Definition}

\theoremstyle{remark}
\newtheorem{remark}[thm]{Remark}

\numberwithin{equation}{section} %{chapter}

\makeatletter  
  \@addtoreset{equation}{section}
\makeatother

\usepackage[tight]{minitoc}               % shrink spacing
\mtcsetfeature{secttoc}{open}{\vspace{-6pt}}  % default=0pt
\mtcsetfeature{secttoc}{close}{\vspace{-3pt}} % default=0pt
\def\comment#1{ }

\author{\large %Akihito Ebisu\\
Yoshishige Haraoka \\
Hiroyuki Ochiai\\
Takeshi Sasaki\\
Masaaki Yoshida}

\title{\bf Fuchsian differential equations of order 3,...,6 with three singular points and an accessory parameter II,  Equations of order 3}

\date{\today}

\begin{document}

\maketitle

%\newpage

%\newpage
 
\begin{abstract}
  Fuchsian differential equations of order 3  with three singular points and with an accessory parameter are studied. When local exponents are generic, no shift operator is found, for codimension-1 subfamilies, neither. We found shift operators for several codimension-2 subfamilies of which accessory
parameter is assigned as a cubic polynomial in the local exponents.
 The Dotsenko-Fateev equation is one of them. 
\end{abstract}
\dosecttoc
\tableofcontents
\vfill

\noindent
{\bf Subjectclass}[2020]: Primary 34A30; Secondary 34M35, 33C05, 33C20, 34M03.


\newpage

\section*{Introduction}%\nostcrule%\secttoc
\addcontentsline{toc}{section}{\protect\numberline{}Introduction}  
%\addcontentsline{toc}{section}{Introduction}
In this paper we study shift operators for Fuchsian differential equations of order 3 with three singular points and one accessory parameter.
\par\medskip
In the previous paper \cite{HOSY1}, we started from the equation of order 3 as above and lifted this equation, via addition and middle convolution, to a differential equation of order 6 with an accessory parameter. While studying this equation, we find that if the accessory parameter is assigned as a cubic polynomial in the local exponents, this equation has nice properties: shift operators and several symmetries.
\par\medskip
We push down this equation, via addition and middle convolution, to a differential equation of order 3, of which accessory parameter is now assigned as a cubic polynomial of the local exponents. This equation, called $E_3$, is the differential equation we study in this paper.
\par\medskip
Though we can not find any shift operator for $E_3$ if the local exponents are generic, we find several codimension-2 conditions (we could not find any codimension-1 condition) on the six local exponents of $E_3$ to define differential equations with four free local exponents admitting shift operators for four independent shifts.  The Dotsenko-Fateev equation is one of them. % two equations admitting shift operators for less than four independent shifts.
\par\medskip
This paper is organized as follows: Section 1 explains how $E_3$ is derived from the equation of order 6 found in \cite{HOSY1}.  \par\noindent
The shift operators of the Dotsenko-Fateev equation are found in Section 2. \par\noindent
The shift operators of the equation related to the equation of order four found and studied in \cite{Z12} are obtained in Section 3. \par\noindent
Section 4 studies the symmetry of the cubic polynomial $A_{00}(e)$ in the local exponents. \par\noindent
Thanks to this symmetry, we find other four codimension-2 restrictions of $E_3$ admitting four independent shift operators. They are presented in Section 5.\par\noindent
There are codimension-2 restrictions of $E_3$ admitting less than four independent shift operators. Two examples are presented in Section 6.\par\noindent
In order to define and study equations, we need several tools of
investigation. We extract some of them from \cite{HOSY1} and put in the last section.
\par\medskip
We acknowledge that we used the software Maple, especially
{\sl DEtools \!}-package for multiplication and division of
differential operators.
Interested readers may refer to our list of data written in text files of Maple format 
\footnote{ http://www.math.kobe-u.ac.jp/OpenXM/Math/FDEdata} 
for the differential equations and the shift operators treated in this document.

\newpage
\section{Equations $H_3$, $E_6$ and $E_3$} \label{$G3$}\setcounter{stc}{2}\secttoc
In this section, we introduce a Fuchsian differential equation $E_3$ of order 3 with three singular points, and with the unique accessory parameter specified as a cubic polynomial of the local exponents. We first recall a very symmetric Fuchsian differential equation $E_6$ of order 6 found in \cite{HOSY1}, and transform it via addition and middle convolution to get $E_3$. 

\subsection{Equation $H_3$}
Any Fuchsian differential equation of order 3 with the Riemann scheme
$$R_3(\epsilon)=\left(\begin{array}{lll}0&\epsilon_1&\epsilon_2\\0&\epsilon_3&\epsilon_4\\ \epsilon_5&\epsilon_6&\epsilon_7\end{array}\right)\quad \epsilon_1+\cdots+\epsilon_7=3
$$
  can be written in $(x,\partial)$-form
  \footnote{See \cite{HOSY1} for the definition of $(x,\partial)$-form and $(x,\theta ,\partial)$-form of differential operators.}
as $$H_3: a_3\partial^3+a_2\partial^2+a_1\partial+a_0,\quad \partial=\frac{d}{dx},$$
where
\[\def\arraystretch{1.2}\setlength\arraycolsep{2pt} \begin{array}{rcl}
  a_3&=&x^2(x-1)^2,\ \ a_2=x(x-1)(a_{21}x+a_{20}),\\
  a_1&=&a_{12}x^2+a_{11}x+a_{10},\ \ a_0=a_{01}x+a_{00}, \\
  a_{21}&=&6-\epsilon_1-\epsilon_2-\epsilon_3-\epsilon_4,\ \ a_{20}=-3+\epsilon_1+\epsilon_2,\\
  a_{12}&=&(\epsilon_5+\epsilon_6+1)\epsilon_7+(\epsilon_6+1)(\epsilon_5+1), \\
  a_{11}&=&-(\epsilon_5+\epsilon_6)\epsilon_7+(-\epsilon_1\epsilon_2+\epsilon_3\epsilon_4-\epsilon_5\epsilon_6+2\epsilon_1+2\epsilon_2-4),\\
% \end{array} \]
% \[\def\arraystretch{1.2}\setlength\arraycolsep{2pt} \begin{array}{rcl}  
  a_{10}&=&(\epsilon_1-1)(\epsilon_2-1),\ a_{01}\ =\ \epsilon_5\epsilon_6\epsilon_7,\  %\epsilon_7\ =\ s,
\end{array}
\]
and in $(x,\theta ,\partial)$-form as 
\footnote{The composition of two differential operators $P$ and $Q$ is denoted by $P\circ Q$, often abbreviated as $PQ$. }
$$xS_n+S_0+S_1\circ\partial,\quad \theta =x\partial,$$
where
\[\def\arraystretch{1.1}\setlength\arraycolsep{2pt} \begin{array}{rcl}
  S_n&=&(\theta +\epsilon_5)(\theta +\epsilon_6)(\theta +\epsilon_7), \\
  S_0&=&-2\theta^3+(2\epsilon_1+2\epsilon_2+\epsilon_3+\epsilon_4-3)\theta^2 \\
  &&\qquad +(-\epsilon_1\epsilon_2+(\epsilon_3-1)(\epsilon_4-1)-\epsilon_5\epsilon_6-\epsilon_6\epsilon_7-\epsilon_7\epsilon_5)\theta +a_{00},\ \\
  S_1&=&(\theta -\epsilon_1+1)(\theta -\epsilon_2+1).
\end{array}  \setlength\arraycolsep{3pt} 
\]
The coefficient $a_{00}$ does not affect the local exponents. In this sense one can call this coefficient the accessory parameter.

\subsection{Equation $E_6$ found in \cite{HOSY1}}
Any Fuchsian differential equation of order 6  with the Riemann scheme
$$R_6:\left(\begin{array}{lcccccc}x=0:&0&1&2&e_1&e_2&e_3\\x=1:&0&1&2&e_4&e_5&e_6\\x=\infty:&s&s+1&s+2&e_7&e_8&e_9\end{array}\right),\quad e_1+\cdots+e_9+3s=6,
  $$
and with spectral type $(3111,3111,3111)$ can be written in $(\theta ,\partial)$-form as
   \begin{eqnarray}\label{eqT}H_6=T_0(\theta )+T_1(\theta )\partial+T_2(\theta )\partial^2+T_3(\theta )\partial^3,\end{eqnarray}
where
{\setlength\arraycolsep{2pt} \def\arraystretch{1.1}
  \begin{eqnarray} 
\quad T_0&=&(\theta +2+s)(\theta +1+s)(\theta +s)B_0,\quad B_0=(\theta +e_7)(\theta +e_8)(\theta +e_9), \label{eqT0}\\
\quad      T_1&=&(\theta +2+s)(\theta +1+s)B_1,\quad B_1=T_{13}\theta^3+T_{12}\theta^2+T_{11}\theta +T_{10}, \label{eqT1}\\
\quad    T_2&=&(\theta +2+s)B_2,\quad B_2=T_{23}\theta^3+T_{22}\theta^2+T_{21}\theta +T_{20}, \label{eqT2}\\
\quad    T_3&=&(-\theta -3+e_1)(-\theta -3+e_2)(-\theta -3+e_3), \label{eqT3}
  \end{eqnarray}
  }
and
\[\def\arraystretch{1.2}\setlength\arraycolsep{3pt} \begin{array}{rcl}
  T_{13}&=&-3,\quad T_{23}=3,\quad  T_{12}=-9+s_{11}-2s_{13},\quad  T_{22}=18+s_{13}-2s_{11},\\
T_{11}&=&-8+(s_{11}^2+2s_{11}s_{13}-s_{12}^2+s_{13}^2)/3+s_{11}-5s_{13}-s_{21}+s_{22}-2s_{23},\\
T_{21}&=&35+(-s_{11}^2-2s_{11}s_{13}+s_{12}^2-s_{13}^2)/3-7s_{11}+5s_{13}+2s_{21}-s_{22}+s_{23},\\
T_{20}&=&-T_{10}+19+(s_{11}^2s_{13}-s_{11}s_{12}^2+s_{11}s_{13}^2-s_{12}^2s_{13})/9+(s_{13}^3+s_{11}^3-2s_{12}^3)/27\\
&&+(-2s_{11}^2-4s_{11}s_{13}+s_{11}s_{22}+2s_{12}^2+s_{22}s_{12}-2s_{13}^2+s_{22}s_{13})/3\\
&&
-5s_{11}+4s_{13}+3s_{21}-2s_{22}-s_{31}-s_{32}-s_{33},
\end{array}
\]
except $T_{10}$, which does not affect the local exponents. In this sense, one can call this coefficient the {\rm accessory parameter}.
Here $s_*$ are symmetric polynomials of the exponents:
\def\arraystretch{1.1}\setlength\arraycolsep{1pt}
\begin{equation}\label{e19tos}\begin{array}{rl}      %%# new linebreak
s_{11}&=e_1+e_2+e_3,\quad s_{12}=e_4+e_5+e_6,\quad s_{13}=e_7+e_8+e_9,\\
s_{21}&=e_1e_2+e_1e_3+e_2e_3,\quad s_{22}=e_4e_5+e_4e_6+e_5e_6,\\ 
s_{23}&=e_7e_8+e_7e_9+e_8e_9,\quad s_{31}=e_1e_2e_3,\quad s_{32}=e_4e_5e_6,\\ 
s_{33}& = e_7e_8e_9,\quad s=-(s_{11}+s_{12}+s_{13}-6)/3.
\end{array}\end{equation}
\setlength\arraycolsep{3pt}
\begin{dfn}The differential equation $H_6$ with the following cubic polynomial $S_{10}$ as the coefficient $T_{10}$ %the  following accessory parameter $T_{10}=S_{10}(e)$
  is called $E_6(e)$. 
    $$\begin{array}{l}
    54S_{10}:=2s_{11}^3 + 6s_{11}^2s_{13} - 2s_{12}^3 - 6s_{12}^2s_{13} + 9s_{11}^2 + 18s_{11}s_{13} 
    - 9s_{11}s_{21} \\+ 18s_{11}s_{23} - 9s_{12}^2 + 9s_{12}s_{22} + 9s_{13}^2     - 18s_{13}s_{21} 
    + 18s_{13}s_{22} + 9s_{13}s_{23} \\+ 18s_{11} - 126s_{13} - 27s_{21} + 27s_{22} - 135s_{23} + 27s_{31} - 27s_{32} - 81s_{33} - 135\\
  \end{array}$$
\end{dfn}

\subsection{Definition of $E_3$ as a middle convolution of $E_6$} \label{E3Def}
The $(\theta ,\partial)$-form of $E_6$ suggests, thanks to the formulae
\[(\theta +3)(\theta +2)(\theta +1)=\partial^3x^3,       %%# spacing
   \ (\theta +3)(\theta +2)\partial=\partial^3x^2,
   \ (\theta +3)\partial^2=\partial^3x,
   \ \theta \partial=\partial(\theta -1),\]
to modify the expression by replacing
$\theta $ by $\theta -t$ (middle convolution with parameter $t$),
where
  \[t:=s-1,\quad s=2-(e_1+\cdots+e_9)/3,\]
so that $E_6(\theta =\theta -t)$ is divisible by $\partial^3$ from the left
and, if we write the quotient by $mcE_6=x^3(x-1)^3\partial^3+\cdots$,
then its Riemann scheme is
$$ \left(\begin{array}{ccc}
  e_1+t&e_2+t&e_3+t\\ e_4+t&e_5+t&e_6+t\\ e_7-t&e_8-t&e_9-t\end{array}\right).
$$
We next make an addition: $mcE_6\circ x^{t+e_3}(x-1)^{t+e_6}$; 
the Riemann scheme changes into
$$ \left(\begin{array}{ccc}
  0&e_1-e_3&e_2-e_3\\ 0&e_4-e_6&e_5-e_6\\
  e_7+e_3+e_6+t&e_8+e_3+e_6+t&e_9+e_3+e_6+t
\end{array}\right).
$$
Introduce parameters $\epsilon_1,...,\epsilon_7$ by
$$\begin{array}{llll}
 &\qquad\qquad\qquad\qquad e_1 - e_3 &= \epsilon_1,\qquad\qquad\qquad  e_2 - e_3 &= \epsilon_2,\\
 &\qquad\qquad\qquad\qquad e_4 - e_6 &= \epsilon_3,\qquad\qquad\qquad  e_5 - e_6 &= \epsilon_4,\\
e_3 + e_6 + e_7+t &= \epsilon_5,\ \qquad e_3 + e_6 + e_8+t &=\epsilon_6,\ \qquad  e_3 + e_6 + e_9+t &= \epsilon_7,
\end{array}$$
where $\epsilon_1+\cdots+\epsilon_7=3$. Then it is the equation $H_3(\epsilon)$, with the accessory parameter $a_{00}$ (the constant term of this equation) replaced by a cubic polynomial $A_{00}(\epsilon)$, where
\begin{equation}\label{accpara00}
  \def\arraystretch{1.1}\setlength\arraycolsep{2pt} \begin{array}{rcl}
  54A_{00}(\epsilon)&=& -4(\epsilon_1+\epsilon_2-\epsilon_3-\epsilon_4)^3-27\epsilon_5\epsilon_6\epsilon_7\\
  &&+9(\epsilon_1+\epsilon_2-\epsilon_3-\epsilon_4)(\epsilon_5\epsilon_6+\epsilon_5\epsilon_7+\epsilon_6\epsilon_7-2)\\
  &&  +9\epsilon_1\epsilon_2(\epsilon_1+\epsilon_2-1)+18(\epsilon_1+\epsilon_2-1)(\epsilon_3^2+\epsilon_3\epsilon_4+\epsilon_4^2)\\
  &&-9\epsilon_3\epsilon_4(\epsilon_3+\epsilon_4-1)-18(\epsilon_3+\epsilon_4-1)(\epsilon_1^2+\epsilon_1\epsilon_2+\epsilon_2^2),
  \end{array}
\end{equation}

In this way, we get a middle convolution of $E_6(e)$, which we call $E_3(\epsilon)$. Changing the notation of the local exponents from $\epsilon$ to $e$, we have

\begin{dfn}$E_3(e)$ is the equation with the Riemann scheme $R_3(e)$
  and with the accessory parameter $a_{00}=A_{00}(e).$
  %+R(e), \quad e=(e_1,\dots,e_7),\ e_1+\cdots+e_7=3$$ defined above. $E_3(e)=G_3(e,0).$
\end{dfn}

Unlike the equation $E_6(e)$, studied in \cite{HOSY1}, for the equation $E_3(e)$,
no shift operator is found if $e=(e_1,\dots,e_7)$ are generic, no codimension-1 reducibility condition is found to the authors.

\subsection{Symmetries of $E_3$}
\begin{prp}\label{symmE3}
  \begin{itemize}
    \item Adjoint symmetry: the adjoint $E_3^*$ of $E_3=a_3\partial^3+a_2\partial^2+a_1\partial+a_0$ is by definition $E_3^*:=\partial^3\circ a_3-\partial^2\circ a_2+\partial\circ a_1-a_0.$
  This is equal to $$E_3^*=E_3(-e_1,\dots,-e_4,2-e_5,2-e_6).$$
\item $(x\to1-x)$-symmetry: $$E_3(e_1,\dots,e_6)|_{x\to1-x}=E_3(e_3,e_4,e_1,e_2,e_5,e_6),$$
\item $(x\to1/x)$-symmetry: $$x^{-s}E_3(e_1,\dots,e_6))|_{x\to1/x}\circ x^{s}=E_3(e_5-s,e_6-s,e_3,e_4,e_1+s,e_2+s),$$
  where $E_3|_{x\to1-x}$ and $E_3|_{x\to1/x}$ are $E_3$ after the coordinate changes $x\to1-x$ and $x\to1/x$, respectively.
\end{itemize}
\end{prp}
\subsection{Self-afjoint $E_3$}
By Proposition \ref{symmE3}, the equation $E_3$ is self-adjoint if and only if
$$e_1=e_2=e_3=e_4=0,\quad e_4=e_5=1.$$
Its Riemann scheme of is 
$$\left(\begin{array}{cccc}x=0&0&0&0\\ x=1&0&0&0\\ x=\infty&1&1&1
        \end{array}\right),$$
      and the accessory parameter is $A=-1/2$. 
It is irreduciblw and is solved by the square of the hypergeometric function $F(1/2,1/2,1;x)^2$.

\section{Equation $S\!E_3$ }\label{SE3}\secttoc
In this section, we first recall the Dotsenko-Fateev equation, of order 3, and find that it is a codimension-2 specialization of $E_3$. We find for this equation  shift operators for four independent shifts, the S-values (see Definition \ref{Svalues}), and the reducible cases.  

\subsection{The Dotsenko-Fateev equation}
The Dotsenko-Fateev equation is originally found as 
a differential equation satisfied by functions in $x$ defined by the integral
\begin{equation}\label{intrep}\int\omega(x),\quad \omega(x):=\prod_{i=1,2} t_i^a(t_i-1)^b(t_i-x)^c\cdot (t_1-t_2)^g\ dt_1\wedge dt_2
\end{equation}
Consider in the real $(t_1,t_2)$-plane the arrangement of seven lines:
  $$ \prod_{i=1,2}t_i(t_i-1)(t_i-x)\cdot(t_1-t_2)=0.$$
Since the number of bounded chambers cut out by this arrangement is 6,
if the exponents of the integrand is generic,
functions defined by the above integral would satisfy
a differential equation of order 6.
But since the integrand is invariant
under the change $t_1\leftrightarrow t_2$,
and the number of bounded chambers modulo this change
is 3, the functions defined by the above integral satisfy
an equation of order $3$, which is the  Dotsenko-Fateev equation.

The Dotsenko-Fateev operator (\cite{DF}) is an operator of order 3
and is defined as
\begin{equation}\label{dfequaion}
  D\!F(a,b,c,g)=x^2(x-1)^2\partial^3+D\!F_1\partial^2+D\!F_2\partial+D\!F_3,
  \end{equation}
where 
\begin{eqnarray*} % \begin{verbatim}
% D\!F_0 &=& (x-1)^2 x^2, \\
 D\!F_1 &=& -(x-1)x(3ax+3bx+6cx+2gx-3a-3c-g), \\
 D\!F_2 &=& 2a^2x^2+4abx^2+12acx^2+3agx^2+2b^2x^2+12bcx^2+3bgx^2\\
 && +12c^2x^2+8cgx^2+g^2x^2-4a^2x-4abx-16acx-4agx+ax^2-8bcx\\
 && -2bgx+bx^2-12c^2x-8cgx +6cx^2 -g^2x+gx^2+2a^2+4ac+ag\\
 && -2ax+2c^2+cg-6cx-gx+a+c, \\
 D\!F_3 &=& c(2a+2b+2c+g+2)(-(2a+2b+4c+2g+2)x+2a+2c+g+1).
\end{eqnarray*} %\end{verbatim}   
The accessory parameter, the constant term of $D\!F_3$,
is $$c(2a+2c+g+1)(2a+2b+2c+g+2)$$ and the Riemann scheme is 
\begin{equation} \label{RDF}
R_{D\!F}=\left( \begin{array}{cccc}
  x=0 & 0&a+c+1&2a+2c+g+2\\
  x=1 & 0&b+c+1&2b+2c+g+2\\
  x=\infty&\ -2c\ &\ -a-b-2c-g-1\ &\ -2a-2b-2c-g-2
\end{array} \right).
\end{equation}

\subsection{Definition of $S\!E_3$}\label{defSE3}
\begin{thm}\label{DFequalsSE3}The Dotsenko-Fateev operator $D\!F$ is equal to $E_3$ with the Riemann scheme $R_{D\!F}$.
\end{thm}
\begin{proof} Substitute
  $$etoag:\quad\begin{array}{lll}e_1&=a+c+1,\ \ e_2&=2a+2c+g+2,\\
    e_3&=b+c+1,\ \ e_4&=2b+2c+g+2,\\
    e_5&=-2c,\ \ \qquad e_6&=-a-b-2c-g-1,\end{array}$$
  in the expression of $A_{00}(e)$ given in \eqref{accpara00}
  to find $A_{00}(etoag)=c(2a+2c+g+1)(2a+2b+2c+g+2)$. \end{proof} % \hfill$\square$ % \bigbreak

\begin{dfn}
  The equation $S\!E_3$ is a specialization of the equation $E_3$ characterized by the system of two equations
  $$2e_1-e_2=2e_3-e_4=-e_5+2e_6-e_7.$$
  When the exponents are parameterized by $\{a,b,c,g\}$ as in $R_{D\!F}$, $S\!E_3$ coincides with the Dotsenko-Fateev equation $D\!F(a,b,c,g)$. The constant term of $S\!E_3$ is $A_{00}(etoag)$.
\end{dfn}

\subsection{Shift operators of $S\!E_3$ and S-values}
\begin{thm}\label{S13shiftopSE3}$S\!E_3$ admits shift operators for the shifts
  $$a\pm:a\to a\pm1,\ b\pm: b\to b\pm1,\ c\pm:c\to c\pm1 {\rm \quad and\quad} g\pm:g\to g\pm2.$$
\end{thm}

Table of the shift operators are in the next subsection. (cf SE3PQ)
\begin{remark}For the shifts $a-,b-,c-$ and $g-$, the shift operators have denominators   such as $x,x-1,x(x-1)$.
  This is to be compared with that for the other equations. See $P_{(1010)}$ in Proposition \ref{shiftopZ3}. 
\end{remark}

Let $P_{a\pm},P_{c\pm},P_{g\pm}$ be the shift operators for the shifts $a\pm$, $c\pm$, $g\pm$, respectively.
  %Since the shift operators have freedom of multiplying constants,
  We normalize them as
\[ \def\arraystretch{1.1}\begin{array}{lll}
    &P_{a-}=(x-1)^2\partial^2+\cdots,  &P_{c-}=\partial^2+\cdots, \\
  &P_{a+}=x^2(x-1)^2\partial^2+\cdots, &P_{c+}=x^2(x-1)^2\partial^2+\cdots, \\&P_{g+}=x^2(x-1)^2\partial^2+\cdots&
\end{array}\]
 and
 $$P_{g-}= ((2c + g)(a + b + g)x^2-(2c + g)(2b + g)x+(2b + g)(a + c + g))\partial^2+\cdots,
 $$
 and the same for the $Q$'s.
 Note that $P_{g-}$ (as well as $Q_{g-}$) has a strange head.

\begin{remark} The shifts $g\to g\pm1$ do not admit shift operators. Since the domains of integration (\ref{intrep}) in the $(t_1,t_2)$-plane is folded along the diagonal, $(t_1-t_2)^g$ should be considered to be $((t_1-t_2)^2)^{g/2}$.\end{remark}

\begin{prp}\label{SvalueSE3}With the normalization above, we get the S-values:% (SG3shift2out):
\begin{eqnarray*}
&& Sv_{a-}:=P_{a+}(a-1)\circ P_{a-} = a(2a+g)(a+1+g+b+c)(2a+2b+2c+2+g), \\
&& Sv_{b-}:=P_{b+}(b-1)\circ P_{b-} = b(2b+g)(a+1+g+b+c)(2a+2b+2c+2+g), \\
%  Svb:=b*(2*b+g)*(2*a+2+2*b+2*c+g)*(a+1+g+c+b):
%&& Sapn:=mult(\subs(a=a+1,P_{a-}),P_{a+}) = (a+1)(2a+2+g)(2a+4+2b+2c+g)(a+2+b+c+g), \\
&& Sv_{c-}:=P_{c+}(c-1)\circ P_{c-} = c(2c+g)(a+1+b+c+g)(2a+2b+2c+2+g), \\
%&& Scpn:=mult(\subs(c=c+1,P_{c-}),P_{c+}) = (c+1)(2c+2+g)(a+2+b+g+c)(2a+4+2b+2c+g),\\
&& Sv_{g-}:=P_{g+}(g-2)\circ P_{g-} = (g-1)(2a+g)(2b+g)(2c+g)(a+b+c+g+1)\\
&& \qquad \qquad \qquad \qquad \times (a+b+c+g)(2a+2b+2c+g+2).
\end{eqnarray*}
Here, for example, $P_{a+}(a-1)\circ P_{a-}$ is an abbreviation of  $P_{a+}(a=a-1)\circ P_{a-}(a)$. \end{prp}
%%%%%%

\subsection{Shift operators of $S\!E_3$ }\label{AppenSE3}
For the shifts $a+,c+,g+$, the operators are expressed in $(x,\theta )$-form:
$$\begin{array}{l}
P=x^2P_{nn}(\theta )+xP_n(x)+P_0(\theta ),\\
Q=x^2Q_{nn}(\theta )+xQ_n(x)+Q_0(\theta ).\end{array}$$
For the shifts $a-,c-,g-$, the operators are expressed in $(x,\partial)$-form:
%\[ x^2(x-1)^2\partial^2 + x(x-1)P1(x)\partial +P2(x),\]
\[\def\arraystretch{1.1} \begin{array}{ll}
    P_{a-}:&   (x-1)^2\partial^2+  (x-1)R_1/x\partial  + R_1/x,\\
    P_{c-}:&     \partial^2    +R_1/(x(x-1))\partial  + R_0/(x(x-1)),\\
    P_{g-}:&  R_2\partial^2+R_3/(x(x-1))\partial  + R_2/(x(x-1)),\\
\end{array}\]
where $R_k$ are used symbolically for polynomials of degree $k$ in $x$ . Note that they have denominators.
\renewcommand{\labelenumi}{(\ref{AppenSE3}.\arabic{enumi})}
\begin{enumerate}
\item $[a+]\quad[e_1+1, e_2+2, e_6-1, e_7-2]$
  \[
  \hskip-36pt
  \begin{array}{lcp{12.8cm}} \hline\noalign{\smallskip}
   P_{nn}  &=& $(\theta +e_5)(\theta +e_6)$, \\
   P_{n}  &=& $-2\theta^2+(4+4a+7c+2g+b)\theta -6ca-6c^2-2cb-3cg-6c$, \\
   P_{0}  &=& $(\theta -e_1)(\theta -e_2)$,\\ % \quad   = (\theta -a-c-1)(\theta -2a-2c-2-g)$ \\
%\hline \noalign{\smallskip}
   Q_{nn}  &=& $(\theta +e_5+1)(\theta +e_6)$, \\
   Q_{n}  &=& $-2\theta^2+(4+4a+7c+2g+b)\theta -6ca-6c^2-2cb-3cg-6c+2+2a+g$, \\
   Q_{0}  &=& $(\theta -e_1)(\theta -e_2-1)$. \\% \quad  = (\theta -a-c-1)(\theta -2a-2c-3-g)$ \\
\hline
\end{array} \]

\vskip0.5pc
\item $[c+]\quad[e_1+1, e_2+2, e_3+1, e_4+2, e_5-2, e_6-2, e_7-2]$
  \[
  \hskip-36pt
  \begin{array}{lcp{12.8cm}} \hline\noalign{\smallskip}
 P_{nn}  &=& $\theta^2+(-5-3a-6c-2g-3b)\theta  +10ca+10cb+3ag+7cg$ \\
  & &\quad $+8+10c^2+18c+2a^2+8a+2b^2+8b+g^2+6g+4ab+3bg$, \\
 P_{n}  &=&$-2\theta^2+(8+6a+3b+9c+3g)\theta  -6cb-14ca-8 $ \\
  & & \quad $-7cg-2bg-4ab-4ag-18c-10c^2-12a-6g-4b-4a^2-g^2$, \\%[-3mm]
 P_{0}  &=& $(\theta -e_1)(\theta -e_2)$, \\
% \hline \noalign{\smallskip}
 Q_{nn}  &=& $P_{nn}[-1]$, \\
%   Q_{nn}=\theta^2+(-2g-7-3b-3a-6c)\theta 
%         +10ca+10cb+3ag+7cg+14+10c^2+24c+2a^2+11a+2b^2+11b+g^2+8g+4ab+3bg:
 Q_{n}  &=& $-2\theta^2+(3g+11+9c+3b+6a)\theta  -6cb-14ca-14$ \\
  & &\quad  $-7cg-2bg-4ab-4ag-24c-10c^2-16a-8g-6b-4a^2-g^2$, \\
 Q_{0}  &=& $(\theta -e_1)(\theta -e_2-1)$. \\
\hline
\end{array}\]
% Note that $P_{n}-Q_{n}=-3\theta +4a+6c+2g+2b+6$.

%\vfill

\item $[g{+}]\quad[e_6-2, e_7-2, e_2+2, e_4+2]$
  \[   \hskip-36pt
  \begin{array}{lcp{12.8cm}} \hline\noalign{\smallskip}
P_{nn} &=& $(k_1\theta +k_0)(\theta +e_5)$, \\
P_{n} &=& $(6g+12+4b+4a+4c)\theta^2 +(-6g^2-12ab-4b^2-16cb-24-8a^2-24cg$\\
    && \quad $-8bg-20ca-24g-46c-12c^2-16ag-32a-18b)\theta $ \\
    &&\quad  $+8cb^2+28cb+36ca+12cbg+20acg+16abc+36c^2$ \\
    &&\quad  $+8c^3+16bc^2+16c^2a+20c^2g+30cg+8ca^2+8cg^2+28c$, \\
P_{0} &=& $(m_1(\theta -1)+m_0)(\theta -e_2)$, \\
% \hline \noalign{\smallskip}
Q_{nn} &=& $(k_1(\theta -1)+k_0)(\theta +e_5+1)$, \\
Q_{n} &=&  $(6g+12+4b+4a+4c)\theta^2  +(-6g^2-12ab-4b^2-16cb-24-8a^2-24cg$\\
  &&\quad $-8bg-20ca-24g -46c-12c^2-16ag-32a-18b)\theta $ \\
  &&\quad $-20+8cb^2-28a+12cbg+26cg-8ab-12ag-4bg+8cg^2$ \\
  &&\quad $+28ca+20c-18g-8b-4g^2-8a^2+36c^2+28cb+16bc^2$\\
  &&\quad $+20c^2g+16abc+20acg+8c^3+8ca^2+16c^2a$, \\
Q_{0} &=& $(m_1\theta +m_0)(\theta -e_2-1)$, \\
\noalign{\smallskip}\hline\noalign{\smallskip}
     && $k_1=-2a-2c-2b-3g-6,\quad    k_0 = 4(a+b+c+g+2)^2+2c(g+1)-g-2$,\\  
     && $m_1=k_1,\quad    m_0 = -2a-2c-8-g^2-2bg-4b-6g$.
\end{array} \]

\item $[a{-}]\quad[e_1-1, e_2-2, e_6+1, e_7+2]$
  \[
  \hskip-36pt
  \begin{array}{lcp{10.7cm}} \hline\noalign{\smallskip}
P_{an} &=& $(x-1)^2 \partial^2  -(x-1)(3a+3b+4c+2g+2-(a+c)/x) \partial$ \\
  && $+ (2a+2b+2c+g+2)(a+b+2c+g+1-c/x)$,\\
Q_{an} &=& $ (x-1)^2 \partial^2  -(x-1)(1+2g+3b+4c+3a-(1+a+c)/x)\partial$ \\
&&  $+(2a+2b+2c+g+1)(a+b+2c+g+1)$\\
&& $-(c+1)(2a+2b+2c+g+2)/x +(a+c+1)/x^2$. \\
 \noalign{\smallskip}\hline
  \end{array} \]
  
\vskip0.5pc
\item $[c-]\quad[e_1-1, e_2-2, e_3-1, e_4-2, e_5+2, e_6+2, e_7+2]$
  \[  \hskip-36pt
  \begin{array}{lcp{10.7cm}} \hline\noalign{\smallskip}
    P_{c-}&=& $\partial^2-(xa+xb+2xc-a-c)/x/(x-1)\,\partial$\\
    && \quad $+c(2a+2+2b+2c+g)/((x-1)x)$, \\
 Q_{c-}&=& $\partial^2-(xb+2xc+2x+xa-a-1-c)/x/(x-1)\,\partial$ \\
&&\quad  $+ ((2+a+2c^2+2ca+4c+cg+b+2cb)x^2$ \\
 &&\quad $+(-2ca-2c^2-2a-4c-cg-2-2cb)x+c+1+a)/x^2/(x-1)^2$. \\
 \noalign{\smallskip}\hline
\end{array} \]

\vskip0.5pc
\item $[g-]\quad[e_6+2, e_7+2, e_2-2, e_4-2]$
  \[
  \hskip-36pt
  \begin{array}{lcp{10.7cm}} \hline\noalign{\smallskip}
  P_{g-}&=& $cp_2\partial ^2+cp_1/(x(x-1))\partial  + cp_0/(x(x-1))$, \\
  Q_{g-}&=& $cq_2\partial ^2+cq_1/(x(x-1))\partial  + cq_0/(x(x-1))^2$,  \\
 \noalign{\smallskip}\hline
\end{array} \]

  \[
  \hskip-36pt
  \begin{array}{lcl}
cp_2 &=& (2c+g)(g+a+b)x^2-(g+2b)(2c+g)x+(g+2b)(g+a+c), \\
cp_1 &=& -(2c+g)(g+a+b)(3a+2g+4c+3b+2)x^3 \\
&&\quad + (2c+g)((a-b)(2+3a -2b +3c)\\
&&\quad + (2b +g)(3+7a +2b+6c +3g))x^2\\
&&\quad  -(g+2b)((a - c) (a + b + 2 c) + (2 c + g) (1 + 3 a + 2 b + 4 c + g))x\\
&& \quad +(a+c)(g+a+c)(g+2b), \\
cp_0 &=& (2a+2+2b+2c+g)\ [(2c+g)(g+a+b)(a+2c+b+1+g)x^2 \\
  &&\quad -(2c+g)((a - b) (1 + a - b + c) - (2 b + g) (1 + 2 a + 2 c + g))x\\
  &&\quad  + c(g+2b)(g+a+c)], \\
%\end{array} \]
%
%  \[
%  \hskip-36pt
%  \begin{array}{lcl}
cq_2 &=& (2c+g)(g+a+b)x^2-(g+2b)(2c+g)x+(g+2b)(g+a+c), \\
cq_1 &=& -(2c+g)(g+a+b)(3a+3b+4c+2g)x^3 +(2c+g)(9cb\\
&& \quad +3ca+6cg+6b^2+3g^2+3a^2-2a+8bg+2b+7ag+9ab)x^2 \\
&& \quad -(g+2b)(g^2+6cg+4g+3ag+2bg+4a+6c^2+3cb+7ca\\
&& \quad +4c+a^2+ab)x + (2+c+a)(g+a+c)(g+2b), \\
cq_0 &=& (2c+g)(g+a+b)(a+g+2c+b)(2a+1+g+2c+2b)x^4 \\
&&\quad -(2c+g)(-2a+10b^2g+14b^2c+6ac^2+10a^2c+8cb+4cg\\
&& \quad +4ab+3ag+5bg+8cg^2+8ag^2+8bg^2+4a^3+12ab^2+21cbg\\
&& \quad +2b+4b^3+2g^2+2g^3+20agb+4b^2+8c^2g+12a^2b+10bc^2\\
&& \quad +10a^2g+24abc+19acg)x^3 +(16a^2cg+16cab^2+8b^2g\\
&& \quad +10b^2c-2ac^2-2a^2c+20cag^2+4b^3c+23cbg^2+22cb^2g\\
&& \quad +32c^2ab+24c^2ag+36c^2bg+40acgb+4c^3a+2a^3g\\
&& \quad +15c^2g^2+16a^2cb+12cb+6cg+12ab+6ag+12bg+11cg^2\\
&& \quad +5ag^2+8bg^2+6ab^2+4ag^3+4ca^3+7cg^3+5b^2g^2+5a^2g^2\\
&& \quad +4bg^3+10c^3g+16c^3b+2b^3g +8c^2a^2+28cbg+g^4+6g^2\\
&& \quad +2g^3+14agb+10c^2g+6a^2b+22bc^2+2a^2g+24abc+10acg\\
&& \quad +20c^2b^2+10abg^2+6a^2bg+6ab^2g)x^2\\
&& \quad -(g+2b)(g+a+c)(4a+2ca+cg+g+2b+2c^2+6c+6+2cb)x \\
&&\quad +(2+c+a)(g+a+c)(g+2b).
\end{array}\]

\end{enumerate}

%%%%%
\subsection{Reducible cases of $S\!E_3$}\label{SE3Red}
When an S-value vanishes the equation is reducible. This leads to 

\begin{thm}\label{redcondSE3} Assume that one of 
\[a,\ b,\ c,\ a+g/2,\ b+g/2,\ c+g/2,\ a+b+c+g/2,\ a+b+c+g,\ (g+1)/2\]
is an integer. Then $S\!E_3$ is reducible. 
\end{thm}

It is known in \cite{Mim} that $S\!E_3$ is irreducible if 
none of the above quantities are not integers and 
$a+c+g/2,\ b+c+g/2\ \notin \frac12+\mathbb{Z}.$

\par\smallskip
We study factorizations of $S\!E_3$ when $a\in\mathbb{Z}$, especially when $a=-1,0$, and when $g\in 2\mathbb{Z}+1$, especially when $g=-1,1$. We observe the shift relations. For the other cases,   the situation is analogous and we skip them.
\subsubsection{$a\in\mathbb{Z}$}
  Writing $D\!F(-1)$ for $D\!F(a=-1)$, $P_{a-}(0)$ for $P_{a-}(a=0)$, etc, we have shift relations
$$(1)\ D\!F(-1)\circ P_{a-}(0)=Q_{a-}(0)\circ D\!F(0),\quad (2)\ D\!F(0)\circ P_{a+}(-1)=Q_{a+}(-1)\circ D\!F(-1),$$
which factorize respectively as
$$(1)\ [1,2]\ [1,1]=[1,1]\ [2,1],\qquad (2)\ [2,1]\ [2]=[2]\ [1,2].$$
\par\smallskip
The equations and the shift operators factor as%We write the factorizations of the left relation explicitly:
\[\begin{array}{ll}D\!F(-1)=Dn_1\circ Dn_2,\ &P_{a-}(0)=P_1\circ P_2,\\[1mm]
&Q_{a-}(0)=Q_1\circ Q_2,\ D\!F(0)=D0_1\circ D0_2,\end{array}\]
where
\[\begin{array}{ll}\def\arraystretch{1.4}\setlength\arraycolsep{2pt}
Dn_1&= (x-1) \partial -(2b+2c+g+1),\\[1mm] %A5
Dn_2&=x(\theta -2c)(\theta -b-2c-g)-(\theta -c)(\theta -2c-g),\\[1mm]%A1
P_1&= (x-1)\partial - (b+2c+g+1) + c/x,\quad
P_2=Dn_1-1,\\[1mm]%A2, A6
Q_1&=Dn_1,\quad
Q_2=P_1+1/x,\\[1mm]%A5, A3
D0_1&=x(\theta -2c)(\theta -b-2c-g-1)-(\theta -c)(\theta -2c-g-1),\\[1mm]%A4
D0_2&=P_2.\end{array}\]
When two operators $R_1$ and $R_2$ are related as 
$$R_1=x^{\mu_0}(x-1)^{\mu_1}\ R_2\circ x^{\nu_0}(x-1)^{\nu_1},$$ 
we write $R_1\sim R_2$. Since
$$D0_1\sim Q_{a+}(-1),\quad D0_2\sim Dn_1,\quad P_{a+}(-1)\sim Dn_2,$$
the relation (2) turns out to be a trivial identity. Since 
$$Dn_1\sim Q_1,\quad P_2\sim D0_2,$$
cancelling them from both sides of (2), we get a relation of type $$[2]\ [1]=[1]\ [2],$$
which is equivalent to a shift relation of the Gauss operator $E_2$.

The factorization above and Theorem \ref{onetwo_to_twoone} lead to 
\begin{prp}\label{SE3Redprop}
    The reducible types of $D\!F$ when $a\in \mathbb{Z}$ are 
    $$\begin{array}{ccccccc}
      a=\  &\ -k&\ -2&\ -1&\ 0&\ 1&\ k\\
  &\ [12]&\ [12]&\ [12]A0&\ [21]A0&\ [21]&\ [21]
    \end{array},$$
    where $[12]$ means the equation factors as $F_1\circ F_2\ ({\rm order}(F_1)=1,{\rm order}(F_2)=2$, and $A0$ means the factors have no singularities out of $\{0,1,\infty\}$.  \end{prp}
\subsubsection{$g\in2\mathbb{Z}+1$}
Writing $D\!F(-1)$ for $D\!F(g=-1)$, $P_{g-}(1)$ for  $P_{g-}(g=1)$, etc, we have shift relations
$$(1)\ D\!F(-1)\circ P_{g-}(1)=Q_{g-}(1)\circ D\!F(1),\quad (2)\ D\!F(1)\circ P_{g+}(-1)=Q_{g+}(-1)\circ D\!F(-1),$$
which factorize as
$$(1)\ [1,2]\ [1,1]=[1,1]\ [2,1],\qquad (2)\ [2,1]\ [2]=[2]\ [1,2].$$
The equations and the shift operators factor as
\[\begin{array}{ll}D\!F(-1)=Dn_1\circ Dn_2,\ &P_{g-}(1)=P_1\circ P_2,\\[1mm]
&Q_{g-}(1)=Q_1\circ Q_2,\ D\!F(1)=D1_1\circ D1_2,\end{array}\]
where
\[\begin{array}{ll}\def\arraystretch{1.4}\setlength\arraycolsep{2pt}
Dn_1&=x^2(x-1)^2\partial-x(x-1)(xa+xb+2xc-a-c-2x+1)
,\\[2mm]
Dn_2&=\partial^2-\displaystyle 2\frac{xa+xb+2xc-a-c}{x(x-1)}\partial+2c\frac{2a+1+2b+2c}{x(x-1)},
\\[2mm]
P_1&=R\partial-\displaystyle\frac{{\rm Poly}_3}{x(x-1)}
,\quad
P_2=\partial-\frac{xa+xb+2xc-a-c+2x-1}{x(x-1)}
,\\[2mm]
Q_1&=R\partial-\displaystyle\frac{{\rm Poly}_3}{x(x-1)},\quad
Q_2=\partial-\frac{{\rm Poly}_3}{x(x-1)R}
,\\[2mm]
D1_1&=x^2(x-1)^2\partial^2-2x(x-1)(xa+xb+2xc-a-c)\partial\\[2mm]
&+2((2c - 1)(a + b + c + 1)x^2+(-2ac - 2bc - 2c^2 + 2a - c + 1)x-1 - c - a)
,\\[2mm]
D1_2&=\partial-\displaystyle\frac{xa+xb+2xc-a-c+2x-1}{x(x-1)}
,\end{array}\]
where
$$R:=(2c + 1)(a + b + 1)x^2-(2c + 1)(2b + 1)x+(2b + 1)(1 + c + a),$$
and ${\rm Poly}_3$ stands symbolically for a polynomial of degree three of $x$.

Since
$$D1_1\sim Q_{g+}(-1),\quad D1_2\sim Dn_1,\quad P_{g+}(-1)\sim Dn_2,$$
the relation (2) turns out to be a trivial identity. We have 
$$ P_2=D1_2,$$
but $Dn_1\not\sim Q_1.$ Cancelling $P_2$ and $D1_2$ from both sides of (2), we get a mysterious relation of type $[1,2]\ [1]=[1,1]\ [2]:$
$$Dn_1\circ Dn_2\circ P_1= Q_1\circ Q_2\circ D1_1.$$
\begin{prp}    The reducible types of $D\!F$ when $g\in 2\mathbb{Z}+1$ are 
    $$\begin{array}{ccccccc}
      g=\  &\ -2k-1&\ -3&\ -1&\ 1&\ 3&\ 2k+1\\
  &\ [12]&\ [12]&\ [12]A0&\ [21]A0&\ [21]&\ [21],
    \end{array}$$
  \end{prp}
  \section{Equation $Z_3$}\label{E3Z3}\secttoc
In this section, we introduce another codimension-2 specialization of $E_3$. We find for this equation  shift operators for four independent shifts,  S-values, and we study the reducible cases.  
\subsection{Definition of $Z_3$}
The equation $Z_3$ is defined as a specialization of $E_3$
by the condition (a system of two equations)
  $$e_1+e_2+e_5=e_3+e_4+e_5=1.$$
Parameterize the 4 ($=6-2$) exponents by $A_0,A_1,A_2,A_3$ as
$$\left(\begin{array}{lccc}x=0:&0&e_1&e_2\\x=1:&0&e_3&e_4\\x=\infty:&e_5&e_6&e_7\end{array}\right)=
  \left(\begin{array}{ccc}0& - A_0 - A_2& A_0 - A_2\\0& -A_1 - A_2& A_1 - A_2\\
    2A_2 + 1& A_3 + A_2 + 1& -A_3 + A_2 + 1\end{array}\right).
$$
    \begin{remark}
      In \cite{Z12}, we encountered the Fuchsian equation $Z(A)=Z(A_0,\dots,A_3)$ of order 4. Define the equation $Z_4(A,k)$ by $\partial^{-k}\circ Z(A)\circ \partial^k$. Then  $Z_4(A,k=A_2+1/2)$ is reducible of type [1,3], and the second factor is our equation $Z_3$.
      \end{remark}
    The equation $$Z_3=x^2(x-1)^2\partial^3+p_2\partial^2+p_1\partial+p_0$$
    is written  as 
\[\def\arraystretch{1.2}\setlength\arraycolsep{2pt} \begin{array}{rcl}
  p_2&=&x(2x - 1)(x - 1)(3 + 2A_2),\\
  p_1&=&(5A_2^2 - A_3^2 + 12A_2 + 7)x^2+(A_0^2 - A_1^2 - 5A_2^2 + A_3^2 - 12A_2 - 7)x\\
  &&\quad -(A_0 + 1 + A_2)(A_0 - 1 - A_2),\\
  p_0&=&(2A_2 + 1)(-A_3 + A_2 + 1)(A_3 + A_2 + 1)x\\
  &&\quad +(2A_2 + 1)(A_0^2 - A_1^2 - A_2^2 + A_3^2 - 2A_2 - 1)/2.\end{array}
\setlength\arraycolsep{3pt}\]
The accessory parameter $a_{00}$ defined in \eqref{accpara00}
is equal to the constant term of $p_0$:
$$a_{00}=(2A_2 + 1)(A_0^2 - A_1^2 - A_2^2 + A_3^2 - 2A_2 - 1)/2.$$

\subsection{Shift operators of $Z_3$}

\begin{thm}\label{Z3shiftop}
  The equation $Z_3$ admits a shift operator for every even   shift
$$(A_0,\dots,A_3)\to (A_0+n_0,\dots,A_3+n_3)\quad n_0+n_1+n_2+n_3\in 2\mathbb{Z}.$$
\end{thm}

The equation $Z_3$  does not have the full symmetry
relative to the parameters $\{A_0,\dots,A_3\}$ that the equation $Z_4$ has.  
Let $P_{(n_0,n_1,n_2,n_3)}$ denote the shift operator for the even shift above.
For $j=0,1,3$, the shift operator for $A_j\to A_j+1$
and that for $A_j\to A_j-1$ are related by changing the sign of $A_j$:
$$P_{(\epsilon_0n_0,\epsilon_1n_1,n_2,\epsilon_3n_3)}=P_{(n_0,n_1,n_2,n_3)}(\epsilon_0A_0,\epsilon_1A_1,A_2,\epsilon_3A_3),\quad n_j=\pm1\ (j=0,1,3).
$$
  Moreover,
  \[\def\arraystretch{1.1}\begin{array}{ll}
    P_{(0101)}(x,A)&=-P_{(1001)}(x',A'), \\
    P_{(01\bar10)}(x,A)&=P_{(10\bar10)}(x',A'),\\
    P_{(0110)}(x,A)&=-P_{(1010)}(x',A')
  \end{array}
  \]
  hold,  where $(x,A)=(x,A_0,A_1,A_2,A_3)$, $(x',A')=(1-x,A_1,A_0,A_2,A_3)$,
  $\bar1=-1$, and $(00\bar11)$
  stands for $(A_0,A_1,A_2,A_3)\to(A_0,A_1,A_2-1,A_3+1)$, etc.
  
Up to this symmetry, we give shift operators for the six shifts
$$(1100),\ (1001),\ (1010),\ (10\bar10),\ (0011),\ (00\bar11).$$

\begin{prp}\label{shiftopZ3}
Let the shift operators $(P_\sigma,Q_\sigma)$, for a shift $\sigma$, are given as
  $$P_\sigma =p_2\partial^2+p_1\partial +p_0,\quad
  Q_\sigma =q_2\partial^2+q_1\partial +q_0,\quad p_2=q_2.$$
The coefficients
  $$\sigma: p_2,\ p_1,\ p_0;\quad p_1-q_1,\ p_0-q_0$$
are given as follows:
\[\def\arraystretch{1.1}\begin{array}{ll}    %%# rewriting
    (1100):& p_2=x(x-1),\\
    &p_1=-(A_0+A_1-2A_2-2)x+A_0-A_2-1,\\
   %&p_0=1/2A_2^2-A_2A_0-A_2A_1+A_2+1/2A_0^2+A_0A_1+1/2A_1^2-1/2A_3^2+1/2;\\
    &p_0=(A_0^2+A_1^2+A_2^2-A_3^2+1)/2+A_0A_1-A_2(A_0+A_1-1),\qquad \\
    &p_1-q_1=2x-1,\\& p_0-q_0=2A_2-A_0-A_1.\\
  (1001):&p_2=x(x-1)^2,\\
    &p_1=(x-1)((3A_2+A_3+3)x+A_0-A_2-1),\\
  %    &p_0=2xA_3A_2+xA_3+2xA_2^2+3xA_2+x-3/2A_2^2+A_2A_0-A_3A_2\\
  %    &\quad-2A_2-1/2+A_0A_3+1/2A_0^2-1/2A_1^2+1/2A_3^2+A_0;\\
  &p_0 =(2A_2+1)(A_2+A_3+1)x +(A_0^2-A_1^2-3A_2^2\\
  &\quad +A_3^2+2A_0A_2-2A_2A_3+2A_0A_3+2A_0-4A_2-1)/2, \\
  &p_1-q_1=-(x-1)(x+1),\\& p_0-q_0=-(2A_2+1)x-A_0-A_3-1.\\
(1010):&p_2=x-1,\\
  %  &p_1=-(A_0^2x-4A_0A_2x-A_1^2x-5A_2^2x+A_3^2x-2A_0^2-4A_0x+2A_2^2\\
  %  &\quad-10A_2x +4A_2-5x+2)/(2x(A_0+1+A_2)),\\
  &p_1=((-A_0^2+A_1^2+5A_2^2-A_3^2+4A_0A_2+4A_0+10A_2+5)x\\
  & \quad -2(A_0+A_2+1)(A_2-A_0+1))/(2x(A_0+1+A_2)),\\
&p_0=-((A_0^2-A_1^2-A_2^2+A_3^2-2A_2-1)(2A_2+1))/(2x(A_0+1+A_2)),\\
&p_1-q_1=(x-2)/x,\\& p_0-q_0=
(2A_2x+A_0-A_2+x+1)/x^2.\\
(10\bar10):&p_2= x^2(x-1)^2,\\
&p_1=x(x-1)(2A_2x+3x-2),\\
   % &p_0=A_2^2x^2-A_3^2x^2+A_0^2x-2A_0A_2x-A_1^2x+A_2^2x+2A_2x^2+A_3^2x\\
   % &\quad-A_0^2+2A_0A_2 +A_0x-A_2^2-3A_2x+x^2-A_0+A_2-x;\\
& p_0= -(A_2+1+A_3)(A_3-1-A_2)x^2 +(A_0^2-A_1^2+A_2^2+A_3^2\\
& \quad -2A_2A_0+A_0-3A_2-1)x -(A_2-A_0)(A_2-A_0-1), \\
&p_1-q_1=0,\\& p_0-q_0= -2(x-1)(A_0+1-A_2).\\
% \end{array} \]
\end{array}\]  

\[\def\arraystretch{1.1}\begin{array}{ll}    %%# rewriting
% \[\def\arraystretch{1.2}\begin{array}{ll}    %%# rewriting
(0011):&p_2=x(A_3+A_2+1)(x-1),\\
% &p_1=3xA_2^2-3/2A_2^2+4xA_3A_2-2A_3A_2+6xA_2-3A_2+A_3^2x-1/2A_3^2\\
% &\quad+4xA_3 -2A_3-1/2A_0^2-3/2+1/2A_1^2+3x,\\
&p_1= (A_3+3A_2+3)(A_2+1+A_3)x \\
& \quad - (A_0^2-A_1^2+3A_2^2+A_3^2+4A_2A_3+6A_2+4A_3+3)/2, \\
&p_0=((2A_2+1)(A_3+A_2+1))/2;\\
&p_1-q_1=-(2x-1)(A_3+A_2+1),\\& p_0-q_0=
-(3A_2)/2-A_3/2-3/2.\\
%\end{array}\]  
%
%\[\def\arraystretch{1.1}\begin{array}{ll}    %%# rewriting
(00\bar11):&
p_2=x^2(x-1)^2,\\
&p_1=x(2x-1)(x-1)(2A_2+1),\\
  % &p_0=3A_2^2x^2+2A_2A_3x^2-A_3^2x^2+A_0^2x-A_1^2x-3A_2^2x-2A_2A_3x\\
  % &\quad+3A_2x^2+A_3^2x-A_3x^2-A_0^2+A_2^2-3A_2x+A_3x;\\
&p_0=-(A_2+1+A_3)(A_3-3A_2)x^2+(A_0^2-A_1^2-3A_2^2+A_3^2\\
&\quad -2A_2A_3-3A_2+A_3)x+(A_2-A_0)(A_2+A_0), \\
&p_1-q_1=0,\\& p_0-q_0=
-2x(x-1)(A_2-A_3-1).
\end{array} \]
\end{prp}
Note that the shift operator for the shift $(1010)$ has poles.

\subsection{S-values and reducibility conditions of $Z_3$}
\begin{prp}For the shift operator $P_\sigma$, write the S-value
  $P_{-\sigma}(\sigma A)\circ P_\sigma$ mod $Z_3$ by $Sv_\sigma$.
  The S-values for the six shifts above are given as follows:
\[\setlength{\arraycolsep}{1pt}
  \def\arraystretch{1.1}\begin{array}{ll} 
    Sv_{(1100)}&=(A_{01\bar2\bar3}+1)(A_{01\bar23}+1)(A_{012\bar3}+1)(A_{0123}+1),\\
    
    Sv_{(1001)}&=(A_{01\bar23}+1)(A_{0\bar123}+1)(A_{0123}+1)(A_{0\bar1\bar23}+1),\\
   
    Sv_{(1010)}&=(2A_2+1)(A_{0123}+1)(A_{0\bar12\bar3}+1)(A_{0\bar123}+1)(A_{012\bar3}+1)
       /(A_0+A_2+1),\\

    Sv_{(10\bar10)}&=(2A_2-1)(A_{0\bar1\bar2\bar3}+1)(A_{01\bar2\bar3}+1)(A_{0\bar1\bar23}+1)(A_{01\bar23}+1)
    %        /4(A_0-A_2+1),\\
     /(A_0-A_2+1),\\    
    Sv_{(0011)}&=(2A_2+1)(A_{0\bar1\bar2\bar3}-1)(A_{01\bar2\bar3}-1)(A_{0\bar123}+1)(A_{0123}+1)
    % /4(A_3+A_2+1),\\
     /(A_3+A_2+1),\\
    Sv_{(00\bar11)}&=(2A_2-1)(A_{01\bar23}+1)(A_{012\bar3}-1)(A_{0\bar1\bar23}+1)(A_{0\bar12\bar3}-1)
    % /4(A_2-A_3-1),
     /(A_2-A_3-1),
  \end{array} \setlength{\arraycolsep}{3pt}
  \]
where $A_{01\bar2\bar3}=A_0+A_1-A_2-A_3$, and so on.
\end{prp}

\begin{thm}\label{Z3redcond} If one of
$$2A_2+1,\quad \epsilon_0 A_0+\epsilon_1A_1+\epsilon_2A_2+A_3+1\quad (\epsilon_0,\epsilon_1,\epsilon_2=\pm1)$$
is an even integer, then the equation $Z_3$ is reducible.
\end{thm}

\subsection{Reducible cases of $Z_3$} 

\begin{prp}\label{Z3redcase}
  When the equation $Z_3$ is reducible, it factors as
\[\def\arraystretch{1.1} \begin{array}{rcccccl}
A_2=\ &\cdots &-3/2 &-1/2   &1/2    &3/2  &\cdots\\
    &\cdots &[21] &[21]A0 &[12]A0 &[12] &\cdots\\ 
\epsilon_0A_0+\epsilon_1A_1+A_2+A_3=\ &\cdots &-3  &-1  &1   &3    &\cdots\\
                                  &\cdots &[12] &[12]A0 &[21]A0 &[21] &\cdots\\
\epsilon_0A_0+\epsilon_1A_1-A_2+A_3=\ &\cdots&-3&-1&1&3&\cdots\\
                                  &\cdots &[21] &[21]A0 &[12]A0 &[12] &\cdots
\end{array}\]
Here $\epsilon_0, \epsilon_1=\pm1$.
When it appears no apparent singular point, 
the factor $[2]$ is equivalent to $E_2$.
\end{prp}
We omit the proof.

\section{Symmetry of the cubic polynomial $A_{00}(e)$}
\secttoc
In this section, we study the cubic polynomial $A_{00}(e)$  given in \eqref{accpara00} with $\epsilon$ replaced by $e$. It is invariant under permutations of exponents at $0,1$ and at $\infty$:
  $$\{e_1,e_2\},\quad \{e_3,e_4\},\quad{\rm and}\quad \{e_5,e_6,e_7\};$$
  which we call {\it obvious symmetry}.
We show that $A_{00}$ is invariant under an action of the symmetry group of degree eight; much bigger symmetry than the obvious one.
The following theorem is suggested by  Eiichi Sato, to whom the authors are grateful:
\begin{thm}Put
$$F_8:=y_1^3+\cdots+y_8^3, \quad{\rm where}\quad y_1+\cdots+y_8=0,$$
and define $hA_{00}(e_0,\dots,e_6)$, the homogenized $A_{00}(e_1,\dots,e_6)$, by
$$hA_{00}:=e_0^3\cdot A_{00}(e_1/e_0,\dots,e_6/e_0).$$
Put 
    $${ytoe}:\quad\begin{array}{ll}
y_1&=-3e_0 - 2e_1 + 4e_2 + 2e_3 + 2e_4,\\
y_2&=-3e_0 - 2e_2 + 4e_1  + 2e_3 + 2e_4,\\
y_3&=\ \ 3e_0 + 2e_4 - 4e_3 - 2e_1 - 2e_2,\\
y_4&=\ \ 3e_0 + 2e_3 - 4e_4 - 2e_1 - 2e_2,\\
y_5&=-9e_0  + 4e_1 + 4e_2  + 2e_3 + 2e_4 + 6e_5,\\
y_6&=-9e_0  + 4e_1 + 4e_2  + 2e_3 + 2e_4 + 6e_6 ,\\
y_7&=\ \ 9e_0  - 4e_3 - 4e_4- 2e_1 - 2e_2 - 6e_5 - 6e_6 ,\\
y_8&=\ \ 9e_0  - 4e_1 - 4e_2 - 2e_3 - 2e_4,\quad y_1+\cdots +y_8=0,
    \end{array}$$
then we have
$$F_8(ytoe)=-2^4\cdot 3^4\cdot hA_{00}.$$
Conversely, put
    $$etoy:\quad\begin{array}{ll}
  e_1&=-2y_1-y_2-3y_3-3y_4-y_5-y_6-y_7,\\
  e_2&=-2y_2-y_1-3y_3-3y_4-y_5-y_6-y_7,\\
  e_3&=-2y_3-y_4-y_5-y_6-y_7,\\
  e_4&=-2y_4-y_3-y_5-y_6-y_7,\\
  e_5&=-y_1-y_2-y_3-y_4-y_6-y_7,\\
  e_6&=-y_1-y_2-y_3-y_4-y_7-y_5,\\
  e_7&=-y_1-y_2-y_3-y_4-y_5-y_6,\\
  e_0&=-2y_1-2y_2-4y_3-4y_4-2y_5-2y_6-2y_7,\quad e_1+\cdots+e_7=3e_0,
\end{array}$$
then we have $$hA_{00}(etoy)=-F_8/6.$$
The obvious symmetry above is translated in terms of the $y$-coordinates as the permutations of
  $$\{y_1,y_2\},\quad\{y_3,y_4\},\quad{\rm and }\quad \{y_5,y_6,y_7\}.$$
\end{thm}
Once the transformation was found then proof is just a computation. The last statement is checked as follows:
If we substitute $e_6$ by $3-(e_1+\cdots+e_6)$ $(=e_7)$ in $y_6=-9e_0 + 4e_1 + 4e_2 + 2e_3 + 2e_4+ 6e_6 ,$ then it changes into $y_7=9e_0 - 2e_1 - 2e_2  - 4e_3 - 4e_4- 6e_5 - 6e_6$.
\par\medskip
$F_8$ is invariant under the permutations of $\{y_1,\dots,y_8\}$. Thus $A_{00}$ admits the action of the symmetry group of degree eight, which we denote by $S_8$.

\begin{dfn}Let $X_5$ be the 5-dimensional subvariety defined by
$$F_8=0,\quad y_1+\cdots+y_8=0$$
in the seven dimensional projective space coordinatized by  $y_1:\cdots:y_8.$
\end{dfn}
\begin{lemma}The singular points of $X_5$ are isolated; they are
$$\varepsilon_1:\cdots:\varepsilon_8,\quad \varepsilon_i=\pm1,\quad \varepsilon_1+\cdots+\varepsilon_8=0.$$
\end{lemma}
The 3-dimensional linear subspaces
$$y_a+y_b=y_c+y_d=y_e+y_f=0,\quad \{a,\dots,f\}\subset\{1,\dots,8\}$$
live in $X_5$. 
\begin{remark}[Private communication with E.~Sato] They are the maximal dimensional linear subspaces in $X_5$. Conjecture:  The maximal dimensional linear subspaces in $$X_n: y_1^3+\cdots+y_{n+3}^3=y_1+\cdots+y_{n+3}=0,\quad n:{\rm odd}$$
in $\mathbf{P}^{n+2}$ is the $S_{n+3}$-orbit of the linear subspace 
$$y_1+y_2=\cdots=y_{n}+y_{n+1}=0.$$
It is proved only when $n=1,3,5,7$.
\end{remark}
The following is easy to show.
\begin{prp}
The four dimensional linear subspaces spanned by one of the 3-dimensional subspaces in $X_5$ and one of the singular points are of two types
$$y_a+y_b=y_c+y_d=0\quad{\rm and }\quad y_a+y_b+y_c+y_d=y_a+y_b+y_e+y_f=0.$$
\end{prp}
\section{Other four specializations: $E_{3a},E_{3b},E_{3c},E_{3d}$}\secttoc
In this section, we find and study the four equations $E_{3a}\dots,E_{3d}$; find shift operators, S-values, and reducible cases.
\subsection{$S_8$-orbits of $S\!E_3$ and $Z_3$}In the previous subsections, we studied the two restrictions of the exponents
$$e_1-e_2-e_6+1-e_3=e_4+e_1+e_6-1-e_3=0\quad{\rm and}\quad 
e_1+e_5+e_2-1=e_3+e_4+e_5-1=0,$$
defining $S\!E_3$ and $Z_3$. In  terms of the $y$-coordinates, these restrictions turn out to be
$$y_1+y_4=y_3+y_6=0\quad{\rm and}\quad y_1+y_2+y_3+y_4=y_3+y_4+y_6+y_7=0.$$
respectively. 
We consider restrictions of $E_3(e)$ corresponding to the two types of the 4-dimensional subspaces above (the $S_8$-orbits of the two subspaces above), 
and look for equations admitting four independent shift operators $(P,Q)$ of type
$$P,Q=\frac{R_{k-2}\cdot dx^2}{x^m(x-1)^m}+\frac{R_{k-1}\cdot dx}{x^{m+1}(x-1)^{m+1}}+\frac{R_{k}}{x^{m+2}(x-1)^{m+2}},\quad k=18,\ m=4,$$
where $R_k$ is used symbolically for a polynomial of degree $k$ in $x$.
Up to the obvious symmetry, we find four codimension-2 restrictions 
$$E_{3a},\quad E_{3b},\quad E_{3c},\quad E_{3d}$$
of $E_3(e)$, defined by the following four subspaces
$$\begin{array}{l} E_{3a}: y_1+y_4=y_2+y_3=0,\\E_{3b}:y_3+y_5=y_4+y_6=0,\\
E_{3c}:y_1+y_5+y_2+y_8=y_1+y_5+y_3+y_4=0,\\ E_{3d}:y_1+y_5+y_3+y_4=y_1+y_5+y_6+y_8=0,\end{array}$$
in terms of the $e$-coordinates, they are
$$\begin{array}{l} E_{3a}: e_3=e_1,\quad e_4=e_2,\\
E_{3b}:e_2 = -e_3 - e_5 - 2e_6 + 3 - e_1,\quad e_4 = e_3 - e_5 + e_6,\\
E_{3c}:e_2 = 2e_1 + e_3 + e_4,\quad e_5 = 1 - e_1 - e_3 - e_4,\\
E_{3d}:e_2 = -e_3/2 - e_4/2 + 3/2 - e_1 - (3e_6)/2,\quad e_5 = e_1 + e_6.\end{array}$$

\subsection{Equation $E_{3a}$}
In this subsection we study a specialization
$E_{3a}=E_{3a}(e_1,e_2,e_5,e_6)$ of $E_3$
with the condition $$e_3=e_1,\ e_4=e_2.$$ The accessory parameter takes the value
$$A_{00}=-e_5e_6e_7. \qquad e_7=s=-(2e_1+2e_2+e_5+e_6-3).$$

\subsubsection{Shift operators of $E_{3a}$}\label{E3E3aShift}

\begin{thm}\label{shiftopE3a}The equation $E_{3a}$ admits a shift operator for every shift in $e_i\to e_i\pm1,\ e_j\to e_j\pm2 \quad (i=1,2,\ j=5,6)$.
\end{thm}
\par\noindent
{\bf Shift operators for $e_1\to e_1\pm1$:} 
The shift relation $E_{3as}\circ P_{1p} = Q_{1p}\circ E_{3a},\ E_{3as}=E_{3a}(e_1=e_1+1)$ is solved by
\[\def\arraystretch{1.2}\setlength\arraycolsep{2pt}
\begin{array}{lcl}
P_{1p}&=&x^2(x-1)^2\partial^2+1/2x(2x-1)(x-1)(e_5+e_6+1)\partial + e_5e_6x(x-1)\\
&&\quad   -e_1(e_5+e_6+2e_1-1)/2, \\
Q_{1p}&=&x^2(x-1)^2\partial^2+1/2x(2x-1)(x-1)(e_5+e_6+3)\partial \\
&&\quad + (e_5+1)(e_6+1)x(x-1) -e_1(e_5+e_6+2e_1+1)/2,
\end{array}
\]
and the shift relation
$E_{3as}\circ P_{1n} = Q_{1n}\circ E_{3a},\ E_{3as}=E_{3a}(e_1=e_1-1)$ by
\[\def\arraystretch{1.1}\setlength\arraycolsep{2pt}
\begin{array}{lcl}
  P_{1n}&=&x(x-1)\partial^2-(2x-1)(e_2-1)\partial-(2e_2+e_7-1)e_7,\\
Q_{1n}&=&x(x-1)\partial^2-(2x-1)(e_2-1)\partial-(2e_2+e_7)(e_7+1).
\end{array}
\]
\par\medskip\noindent
{\bf Shift operators for $e_5\to e_5\pm2$:} 
The shift relation
$E_{3as}\circ P_{5pp}=Q_{5pp}\circ E_{3a},\ E_{3as}=E_{3a}(e_5=e_5+2)$
for the shift $e_5\rightarrow e_5+2$ is solved by
\[\def\arraystretch{1.2}\setlength\arraycolsep{2pt}
\begin{array}{rcl}
P_{5pp}&=& x^2(x-1)^2\partial^2 + (1/2)x(2x-1)(x-1)(1+e_5+e_6)\partial \\
&&     +e_5e_6x(x-1)+(1/4)e_5(2-e_7-2e_1)(2-e_7-2 e_2)/(2-e_7+e_5),\\
Q_{5pp}&=&x^2(x-1)^2\partial^2+(1/2)x(2x-1)(x-1)(e_5+e_6+5)\partial \\
&& + (e_5+3)(e_6+1)x(x-1)+ (1/16)\left\{(e_5+2)e_7^2+2(e_5+2)(e_1+e_2-3)e_7
\right. \\
&&\left.  +4(4+e_1e_2e_5+3e_5-2(e_1+e_2)(e_5+1))\right\}/(2-e_7+e_5),
\end{array}
\]
and the shift relation
$E_{3as}\circ P_{5nn} = Q_{5nn}\circ E_{3a},\  E_{3as}=E_{3a}(e_5=e_5-2)$
for the shift $e_5\rightarrow e_5-2$ by
\[\def\arraystretch{1.2}\setlength\arraycolsep{2pt}
\begin{array}{lcl}
  P_{5nn}&=& x^2(x-1)^2\partial^2 -1/2x(2x-1)(x-1)(-4+2e_1+2e_2+e_5)\partial \\
 &&   +e_7e_6x(x-1)+(1/4)e_7(2e_2+e_5-2)(2e_1+e_5-2)/(2+e_7-e_5), \\
 Q_{5nn}&=&x^2(x-1)^2\partial^2 -1/2x(2x-1)(x-1)(-4+2e_1+2e_2+e_5)\partial\\
 & &   +(e_7+3)(e_6+1)x(x-1)-(1/4)\left\{(2(4-e_5)(e_1+e_2)-e_5^2+6e_5 \right.\\
&& \left. -4e_1e_2-12)e_7+2(e_5-2)(e_7+e_6+1)\right\}/(2+e_7-e_5),
\end{array} \setlength\arraycolsep{3pt}
\]

\subsubsection{S-values and reducibility conditions  of $E_{3a}$}\label{E3E3aS}

\begin{prp}\label{SvalueE3a}The S-values for the shifts above:
\[\def\arraystretch{1.1} \begin{array}{lcl}
Sv_{1p}&=&P_{1n}(e_1=e_1+1)\circ P_{1p}\\
&=& -(2e_1+e_6-2)(2e_1+e_5-2)(2e_1+e_5+e_6-3)e_7/4, \\
Sv_{5pp}&=& P_{5nn}(e_5=e_5+2)\circ P_{5pp}\\
&=& e_5(2e_1+e_5)(2e_2+e_5)(2-e_7)(2e_1+e_7-2)(2e_2+e_7-2)\\
&&\qquad /(16(2-e_7+e_5)^2).\\
\end{array}\]
\end{prp}

Hence, we have the following theorem.

\begin{thm}\label{redcondE3a}$E_{3a}$ is reducible if one of
  \[\begin{array}{c}
  e_j,\quad 2e_i+e_j,\quad 2e_i+e_{56}+1,\quad 2e_{12}+e_{56}+1\quad(i=1,2,\ j=5,6)
  \end{array}\]
$ (e_{12}=e_1+e_2,\ e_{56}=e_5+e_6)$  is an even integer.
\end{thm}

\subsubsection{Reducible cases  of $E_{3a}$}\label{E3E3aRed}
\begin{prp}\label{factorE3a}When the equation is reducible, it factors as
\[\def\arraystretch{1.1} \begin{array}{rcccccc}
    e_5,\ 2e_1+e_5=\ &\cdots  &-2 &0 &2 &4 &\cdots \\
                    &\cdots &[21] &[21]A0 &[12]A0 &[12] &\cdots \\
    2e_1+e_{56}+1,\ 2e_{12}+e_{56}+1=\ &\cdots  &-2 &0 &2 &4 &\cdots \\
    &\cdots &[12] &[12]A0 &[21]A0 &[21] &\cdots
  \end{array}\]
  When apparent singular point does not appear,
  the factor $[2]$ is equivalent to $E_2$.
\end{prp}
We omit the proof. For $E_{3b}$,..., when the equation is reducible, it factors  of type $\{1,2\}$. We do not give details.

\subsection{Equation $E_{3b}$}\label{E3E3b}
In this subsection we study a specialization $E_{3b}=E_{3b}(e_1,e_3,e_5,e_6)$ of $E_3$
with the condition $$e_2 = -e_3 - e_5 - 2e_6 + 3 - e_1,\quad e_4 = e_3 - e_5 + e_6.$$
The accessory parameter takes the value
$$A_{00}=-((e_3 - e_5)(e_1 - 1)(e_1 + e_3 + e_5 + 2e_6 - 2))/2.$$

\subsubsection{Shift operators of $E_{3b}$}\label{E3E3bShift}

\begin{thm}\label{shiftopE3b}The equation $E_{3b}$ admits a shift operator for every shift  $e_i\to e_i\pm2,\ e_6\to e_6\pm1 \quad (i=1,3,5)$.
\end{thm}
\par\noindent
Since the operators $P$ and $Q$ are very long, we give only the coefficients of $\partial^2$, and the denominators of the other coefficients. For full expression, see $E3bPQ.txt$ in $F\!D\!Edata$ mentioned in the end of Introduction.
\par\noindent
{\bf Shift operators $P_{1pp}, Q_{1pp},P_{1nn}, Q_{1nn}$ for $e_1\to e_1\pm2$:}
$$\begin{array}{l}
  P_{1pp},\ P_{1nn}:(x-1)^2\partial^2+\displaystyle\frac{R_1(x-1)}{x}\partial+\frac{R_1}{x},\\[2mm]
  Q_{1pp},\ Q_{1nn}:(x-1)^2\partial^2+\displaystyle\frac{R_1(x-1)}{x}\partial+\frac{R_2}{x^2},\\
\end{array}$$
where $R_k$ is used symbolically for a polynomial of degree $k$ in $x$.
\par\noindent
 {\bf Shift operators $P_{3pp}, Q_{3pp},P_{3nn}, Q_{3nn}$ for $e_3\to e_3\pm2$:}
$$\begin{array}{l}
  P_{3pp}:(x-1)^2R_2\partial^2+\displaystyle\frac{(x-1)R_3}{x}\partial+\frac{R_3}{x},\\[2mm]
  Q_{3pp}:(x-1)^2R_2^2\partial^2+\displaystyle\frac{(x-1)R_3}{x}\partial+\frac{R_4}{x^2},\\[2mm]
  P_{3nn}:x^2\partial^2+\displaystyle\frac{xR_1}{x-1}+\frac{R_1}{x-1},\\[2mm]
  Q_{3nn}:x^2\partial^2+\displaystyle\frac{xR_1}{x-1}+\frac{R_2}{(x-1)^2},
  \end{array}$$
 where $R_k$ is used symbolically for a polynomial of degree $k$ in $x$.
 \par\noindent
 {\bf Shift operators $P_{5pp}, Q_{5pp},P_{5nn}, Q_{5nn}$ for $e_5\to e_5\pm2$:}
$$\begin{array}{l}
  P_{5pp}:R_2\partial^2+\displaystyle\frac{R_3}{(x-1)x}\partial+\frac{R_2}{(x-1)x},\\[2mm]   
  Q_{5pp}:R_2\partial^2+\displaystyle\frac{R_3}{(x-1)x}\partial+\frac{R_4}{(x-1)^2x^2},\\[2mm]
  P_{5nn},\ Q_{5nn}:x^2(x-1)^2\partial^2+x(x-1)R_1\partial+R_2,  
 \end{array}$$
 where $R_k$ is used symbolically for a polynomial of degree $k$ in $x$.
 \par\noindent
     {\bf Shift operators $P_{6p}, Q_{6p},P_{6n}, Q_{6n}$ for $e_6\to e_6\pm1$:}
$$\begin{array}{ll}
  P_{6p}&=(x-1)^2\partial^2-1/2(x-1)(e_1x+2e_3x-2e_5x-2e_1-3x+2)/x\partial\\&
  -1/2(e_1e_6x+2e_3e_6x-2e_5e_6x+2e_6^2x+e_1e_3-e_1e_5-e_6x-e_3+e_5)/x,\\

  Q_{6p}&=(x-1)^2\partial^2-1/2(x-1)(e_1x+2e_3x-2e_5x-2e_1-x-2)/x\partial\\&
  -1/2(e_1e_6x^2+2e_3e_6x^2-2e_5e_6x^2+2e_6^2x^2+e_1e_3x-e_1e_5x+e_6x^2+2e_1x\\&
  +e_3x-e_5x-2e_1+2x-2)/x^2,\\

  P_{6n}&= (x-1)x^2\partial^2-(e_3x-2e_5x+2e_5-x)x\partial-e_3e_5x+e_5^2x+e_1^2\\&
  +2e_1e_3+4e_1e_6+e_3^2+4e_3e_6-e_5^2+4e_6^2-5e_1-5e_3+e_5-10e_6+6,\\

  Q_{6n}&= (x-1)x^2\partial^2-(e_3x-2e_5x+2e_5-x)x\partial-e_3e_5x+e_5^2x+e_1^2\\&
  +2e_1e_3+4e_1e_6+e_3^2+4e_3e_6-e_5^2+4e_6^2-7e_1-7e_3+e_5-14e_6+12.
\end{array}$$     

     \subsubsection{S-values and reducibility conditions of $E_{3b}$}\label{E3E3bS}
     % E3bSv   Svalues of E3b
\begin{prp}\label{SvalueE3b}The numerators of the S-values for the shifts above:
\[\def\arraystretch{1.1} \begin{array}{lcl}
  Sv_{1pp}&=&(-1+e_1+2e_6)(e_1+2e_5-1)(e_1-1+2e_3+2e_6)(e_1+e_5+e_3)(-e_3+e_1+e_5)\\
  &&\times(e_1+e_3+2e_6-e_5),\\
 %&& /(e_3+2e_1+e_5+2e_6-1)^2,\\
%
  Sv_{3pp}&=&(e_3+e_5)(e_3+2-e_5)(e_3+2e_6-e_5)(e_1+1+2e_3+2e_6)(e_1-1+2e_3+2e_6)\\
  &&\times(e_1+e_5+e_3)(e_1-2+e_5-e_3)(e_1+e_3+2e_6-e_5),\\
% && /((e_1e_3+2e_1-e_1e_5+3e_3^2+4e_3-2e_3e_5+4e_3e_6+4e_6-e_5^2)(e_3+1)(e_3-e_5+e_6+1)),\\
%
  Sv_{5pp}&=&(e_3+e_5)(e_3-e_5-2+2e_6)(e_1+2e_5+1)(e_1+2e_5-1)(e_1+e_5+e_3)\\
  &&\times(-e_3+e_1+e_5)(e_1+e_3-e_5-2+2e_6)(e_3-e_5),\\
% && /((e_1e_3-e_1e_5-2e_1+2e_1e_6+e_3^2+2e_3e_5-3e_5^2-4e_5+4e_5e_6)(e_5-e_6+1)(e_3-e_5+e_6-1)),\\
%
Sv_{6p}&=&(e_3+2e_6-e_5)(-1+e_1+2e_6)(e_1-1+2e_3+2e_6)(e_1+e_3+2e_6-e_5).
\end{array}\]
\end{prp}

Note that the factors of the numerators of the S-values above are $\Bbb{Z}$-linear forms in $$1,\quad e_1,\quad e_3,\quad e_5,\quad 2e_6.$$
Hence, we have the following theorem.

\begin{thm}\label{redcondE3b}
Let $\varphi(e_1,e_3,e_5,e_6)$ be one of the factors of the numerators of the S-values above. If $\varphi$ is an even integer, then $E_{3b}$ is reducible.
\end{thm}
\subsubsection{Detour}Expressions of the shift operators for the shifts $e_i\to e_i\pm2\ (i=1,3,5)$ are fairly long. But those for the shifts $(e_3,e_5)\to (e_3\pm1,e_5\pm1), (e_1,e_6)\to(e_1\pm2,e_6\mp1)$ have relatively short expressions (see $E3cPQ.txt$  in the list of $F\!D\!Edata$). They together with the shift operators for $e_6\to e_6\pm1$ generate the shifts  $e_i\to e_i\pm2\ (i=1,3,5)$. For example, 
$$\begin{array}{ll}
  P_{3p5p}&=(x-1)^2\partial^2-1/2(x-1)(e_1x+2e_3x-2e_5x-2e_1-3x+2)/x\partial\\&
  -1/2(e_1e_5x+2e_3e_5x+e_1e_3-e_1e_5-e_5x-e_3+e_5)/x,\\

P_{3n5n}&=(x-1)x^2\partial^2-(e_3x-e_5x-e_6x+2e_6-x)x\partial-e_3e_6x+e_5e_6x
+e_1^2+2e_1e_3\\&+2e_1e_5+2e_1e_6+e_3^2+2e_3e_5+2e_3e_6+e_5^2+2e_5e_6-5e_1-5e_3-5e_5-4e_6+6,\\

P_{1pp6n}&=(x-1)x^2\partial^2-(e_3x-2e_5x+2e_5-x)x\partial-e_3e_5x+e_5^2x+e_1^2+2e_1e_5-e_1,\\
& \\
S{v3p5p}&=(e_3+e_5)(e_1+2e_5-1)(e_1+e_3+e_5)(e_1+2e_3+2e_6-1),\\

S{v3p5n}&=(e_3+2-e_5)(e_3-e_5+2e_6)(e_1+2e_5-3)(e_1+2e_3+2e_6-1)\\&
\times(e_1-e_3+e_5-2)(e_1+e_3-e_5+2e_6)/(e_3-e_5+e_6+1)^2,\\

S{v1pp6n}&=(e_1+2e_5-1)(e_1+e_3+e_5)(e_1-e_3+e_5)(e_3-e_5+2e_6-2).

\end{array}$$

\subsection{Equation $E_{3c}$}\label{E3E3c}
In this subsection we study a specialization $E_{3c}=E_{3c}(e_1,e_3,e_4,e_6)$ of $E_3$
with the condition $$e_2 = 2e_1 + e_3 + e_4, \quad e_5 = 1 - e_1 - e_3 - e_4.$$ The accessory parameter takes the value
$$A_{00}=-((2e_1 + e_3 + e_4 - 1)(2e_1e_6 + e_3e_4 + e_3e_6 + e_4e_6 + e_6^2 - 2e_6))/2.$$

\subsubsection{Shift operators of $E_{3c}$}\label{E3E3cShift}
\begin{thm}\label{shiftopE3c}The equation $E_{3c}$ admits a shift operator for every shift  $e_1\to e_1\pm1,\ e_i\to e_i\pm2\quad (i=3,4,6)$.
\end{thm}
\par\noindent
We give only the coefficients of $\partial^2$, and the denominators of the other coefficients. For full expression, see $E3cPQ.txt$  in the list of $F\!D\!Edata$.
\par\noindent
{\bf Shift operators $P_{1p}, Q_{1p},P_{1n}, Q_{1n}$ for $e_1\to e_1\pm1$:}
$$\begin{array}{l}
  P_{1p},\ Q_{1p}:x^2(x-1)^2\partial^2+x(x-1)R_1\partial+R_2,\\[2mm]
  P_{1n}::(x-1)^2\partial^2+\displaystyle\frac{(x-1)R_1}{x}+\frac{R_1}{x},\\[2mm]
  Q_{1n}::(x-1)^2\partial^2+\displaystyle\frac{(x-1)R_1}{x}+\frac{R_2}{x^2},
\end{array}$$
where $R_k$ is used symbolically for a polynomial of degree $k$ in $x$.
\par\noindent
{\bf Shift operators $P_{3pp}, Q_{3pp},P_{3nn}, Q_{3nn}$ for $e_3\to e_3\pm2$:}
$$\begin{array}{l}
  P_{3pp},\ Q_{3pp}:x^2(x-1)^e\partial^2+x(x-1)R_1\partial+R_2,\\
  P_{3nn}:R_2\partial^2+\displaystyle\frac{R_3}{x(x-1)}+\frac{R_2}{x(x-1)},\\
  Q_{3nn}:R_2\partial^2+\displaystyle\frac{R_3}{x(x-1)}+\frac{R_4}{x^2(x-1)^2},\\
\end{array}$$
where $R_k$ is used symbolically for a polynomial of degree $k$ in $x$. The shift operators for $e_4\to e_4\pm2$ are similar to those for $e_3\to e_3\pm2$.
\par\noindent
{\bf Shift operators $P_{6pp}, Q_{6pp},P_{6nn}, Q_{6nn}$ for $e_6\to e_6\pm2$:}
$$\begin{array}{l}
  P_{6pp},\ Q_{6pp},\ P_{6nn},\ Q_{6nn}:x^2(x-1)^2\partial^2+x(x-1)R_1\partial+R_2,  
  \end{array}$$
where $R_k$ is used symbolically for a polynomial of degree $k$ in $x$.

\subsubsection{S-values and reducibility conditions of $E_{3c}$}\label{E3E3cS}
% E3cSv  
\begin{prp}\label{SvalueE3c}The numerators of the S-values for the shifts above:
  \[\def\arraystretch{1.1} \begin{array}{lcl}
Sv_{1p}&=&(2e_1+e_4+e_6)(2e_1+e_3+e_6)(2e_1+2e_3+e_4+e_6)(2e_1+e_3+2e_4+e_6),\\

Sv_{3pp}&=&(e_3+e_6)(e_3+e_4+1)(2e_1+e_3+e_6)(2e_1+2e_3+e_4+e_6+2)\\
&&\times(2e_1+2e_3+e_4+e_6)(2e_1+e_3+2e_4+e_6),\\

%&&/(2e_1+3e_3+2e_4+e_6+2)/(e_1+e_3+e_4+e_6):\\

Sv_{6pp}&=&(e_4-e_6)(e_4+e_6)(e_3-e_6)(e_3+e_6)(2e_1+e_4+e_6)(2e_1+e_3+e_6)\\
&&\times(2e_1+2e_3+e_4+e_6)(2e_1+e_3+2e_4+e_6).

%&&/(e_1+e_3+e_4+e_6)/(2e_1+e_3+e_4+2e_6)^2/(e_1+e_6):
\end{array}\]
\end{prp}
Note that the factors of the numerators of the S-values above are $\Bbb{Z}$-linear forms in $$1,\quad 2e_1,\quad e_3,\quad e_4,\quad e_6.$$
Hence, we have the following theorem.

\begin{thm}\label{redcondE3c}
Let $\varphi(e_1,e_3,e_4,e_6)$ be one of the factors of the numerators of the S-values above. If $\varphi$ is an even integer, then $E_{3c}$ is reducible.
\end{thm}

\subsection{Equation $E_{3d}$}\label{E3E3d}
In this subsection we study a specialization $E_{3d}=E_{3d}(e_1,e_3,e_4,e_6)$ of $E_3$ 
with the condition $$e_2 = -e_3/2 - e_4/2 + 3/2 - e_1 - (3e_6)/2,\quad e_5 = e_1 + e_6.$$
  The accessory parameter takes the value
$$A_{00}=e_6(2e_1^2 + e_1e_3 + e_1e_4 + 3e_1e_6 + e_3e_4 + e_3e_6 + e_4e_6 + e_6^2 - 3e_1 - e_3 - e_4 - 3e_6 + 1)/2.$$

  \subsubsection{Shift operators of $E_{3d}$}\label{E3E3dShift}
\begin{thm}\label{shiftopE3c}The equation $E_{3d}$ admits a shift operator for every shift  $e_1\to e_1\pm1,\ e_i\to e_i\pm2\quad (i=3,4,6)$.
\end{thm}
\par\noindent
We give only the coefficients of $\partial^2$, and the denominators of the other coefficients. For full expression, see $E3dPQ.txt$ in the list of $F\!D\!Edata$.
\par\noindent
{\bf Shift operators $P_{1p}, Q_{1p},P_{1n}, Q_{1n}$ for $e_1\to e_1\pm1$:}
$$\begin{array}{l}
  P_{1p},\ Q_{1p},\ P_{1n},\ Q_{1n}: x(x-1)^2\partial^2+(x-1)R_1\partial+R_1,
\end{array}$$
where $R_k$ is used symbolically for a polynomial of degree $k$ in $x$.
\par\noindent
{\bf Shift operators $P3p, Q3p,P3n, Q3n$ for $e_3\to e_3\pm2$:}
$$\begin{array}{l}
  P_{3pp},\ Q_{3pp}: x(x-1)^2\partial^2+(x-1)R_1\partial+R_1,\\[2mm]
  P_{3nn}:x^2\partial^2+\displaystyle\frac{xR_1}{x-1}\partial+\frac{R_1}{x-1},\\[2mm]
  Q_{3nn}:x^2\partial^2+\displaystyle\frac{xR_1}{x-1}\partial+\frac{R_2}{(x-1)^2},
\end{array}$$
where $R_k$ is used symbolically for a polynomial of degree $k$ in $x$. 
The shift operators for $e_4\to e_4\pm2$ are similar to those for $e_3\to e_3\pm2$.
\par\noindent
{\bf Shift operators $P_{6p}, Q_{6p},P_{6n}, Q_{6n}$ for $e_6\to e_6\pm2$:}
$$\begin{array}{l}
  P_{_{6p}p}:\displaystyle\frac{(x-1)^2R_2}{x}\partial^2+\frac{(x-1)R_3}{x^2}+\frac{R_3}{x^2},\\[2mm]
  Q_{_{6p}p}:\displaystyle\frac{(x-1)^2R_2}{x}\partial^2+\frac{(x-1)R_3}{x^2}+\frac{R_4}{x^3},\\[2mm]
  P_{6nn},\ Q_{6nn}:x^2(x-1)^2\partial^2+x(x-1)R_1\partial+R_2,
  \end{array}$$
where $R_k$ is used symbolically for a polynomial of degree $k$ in $x$.

\subsubsection{S-values and reducibility conditions of $E_{3d}$}\label{E3EdaS}
% Svalues of E3d
\begin{prp}\label{SvalueE3d}The numerators of the S-values for the shifts above:
  \[\def\arraystretch{1.1} \begin{array}{lcl}
Sv_{1p}&:&(2e_1+e_6)(2e_1+e_4+e_6)(2e_1-1+e_4+2e_6)(2e_1+e_3+e_6)(2e_1+2e_6+e_3-1)\\
&&\times(2e_1-1+2e_6+e_3+e_4),\\

Sv_{3pp}&:&(e_6+e_3)(2e_1+2e_6+e_3-1)(2e_1+e_3+e_6)(2e_1-1+2e_6+e_3+e_4),\\

Sv_{6pp}&:&e_6(e_4+e_6)(e_6+e_3)(2e_1+e_6)(2e_1+e_4+e_6)(2e_1+1+2e_6+e_4)\\
&&\times(2e_1-1+e_4+2e_6)(2e_1+e_3+e_6)(2e_1+1+e_3+2e_6)(2e_1+2e_6+e_3-1)\\
&&\times(2e_1+1+e_3+2e_6+e_4)(2e_1-1+2e_6+e_3+e_4),\\
\end{array}\]
\end{prp}
Note that the factors of the numerators of the S-values above are $\Bbb{Z}$-linear forms in $$1,\quad 2e_1,\quad e_3,\quad e_4,\quad e_6.$$
Hence, we have the following theorem.

\begin{thm}\label{redcondE3d}
Let $\varphi(e_1,e_3,e_4,e_6)$ be one of the factors of the numerators of the S-values above. If $\varphi$ is an even integer, then $E_{3d}$ is reducible.
\end{thm}

\section{Other two specializations: $E_{3e},E_{3f}$}\secttoc
We found two specializations admitting shift operators which are not in the $S_8$ orbits of
$$y_1+y_4=y_3+y_6=0\quad{\rm nor}\quad y_1+y_2+y_3+y_4=y_3+y_4+y_6+y_7=0.$$
The six specializations we treated in the previous sections, they have shift operators for four independent shifts. But the two equations in this section have less than four independent shifts. %For the two equations, we 
\subsection{Equation $E_{3e}: \{e_3=e_1-e_2,e_4=-e_2\}$}
In $y$-coordinates: $\{y_1=-y_4,y_6=-2y_4-2y_3-y_5-y_7-y_2\}$.\par\noindent
\begin{thm}The equation $E_{3e}$ admits a shift operator for every shift $e_1\to\pm1$, $e_j\to \pm2$ ($j=5,6$),\end{thm}
%\begin{remark}We could not find a shift operator for any shift $e_2\to e_2+n$.\end{remark}
\noindent
{\bf Shift operator for $e_1\to e_1+1$}:
$$\begin{array}{ll}
P&=x^2(x-1)^2dx^2+x(x-1)((e_5+e_6+1)x-1/2+e_2-e_5/2-e_6/2)dx\\
&+e_5e_6x^2+(-1/2e_2+1/2e_5e_2+1/2e_6e_2-e_5e_6)x\\
&+e_1/2-e_1^2+e_1e_2-e_5e_1/2-e_6e_1/2,\\
Q&=x^2(x-1)^2dx^2+x(x-1)((e_5+e_6+3)x+e_2-e_5/2-e_6/2-3/2)dx\\
&+(e_5e_6+e_5+e_6+1)x^2+(1/2e_5e_2+1/2e_6e_2-e_5e_6+1/2e_2-e_5-e_6-1)x\\
&-e_1^2+e_1e_2-e_5e_1/2-e_6e_1/2-e_1/2
  \end{array}$$
S-value:
$$(2e_1+e_5+e_6-1)(2e_1-e_2+e_6)(2e_1-e_2+e_5)(2e_1-2e_2+e_5+e_6-1).$$
\noindent
{\bf Shift operator for $e_5\to e_5+2$}:
$$\begin{array}{ll}
P&=x^2(x-1)^2dx^2+x(x-1)((e_5+e_6+1)x-1/2+e_2-e_5/2-e_6/2)dx+e_5e_6x^2\\
&+(-1/2e_2+1/2e_5e_2+1/2e_6e_2-e_5e_6)x-(2e_1e_2e_5+2e_1e_2e_6-2e_1e_5^2\\
&-2e_1e_5e_6-2e_2^2e_5-2e_2^2e_6+3e_2e_5^2+4e_2e_5e_6+e_2e_6^2-e_5^3\\
&-2e_5^2e_6-e_5e_6^2-2e_1e_2+2e_1e_5+2e_2^2-4e_2e_5-2e_2e_6+2e_5^2+2e_5e_6+e_2-e_5)\\
&/(4(2e_1-e_2+2e_5+e_6-1)),\\
Q&=x^2(x-1)^2dx^2+x(x-1)((e_5+e_6+5)x+e_2-e_5/2-e_6/2-5/2)dx+(e_5e_6\\
&+e_5+3e_6+3)x^2+(-3+1/2e_5e_2-e_5e_6+1/2e_6e_2-3e_6-e_5+3/2e_2)x\\
&-(2e_1e_2e_5+2e_1e_2e_6-2e_1e_5^2-2e_1e_5e_6-2e_2^2e_5-2e_2^2e_6\\
&+3e_2e_5^2+4e_2e_5e_6+e_2e_6^2-e_5^3-2e_5^2e_6-e_5e_6^2+6e_1e_2-2e_1e_5\\
&-4e_1e_6-2e_2^2+6e_2e_5+4e_2e_6-2e_5^2-4e_5e_6-2e_6^2-4e_1-e_2-3e_5+2)\\
&/(4(2e_1-e_2+2e_5+e_6-1)),
\end{array}$$
S-value:
$$(e_2 - e_5)(e_2 + e_5)(e_5 + e_6 - 1)(2e_1 + e_5 + e_6 - 1)(2e_1 - e_2 + e_5)(2e_1 - 2e_2 + e_5 + e_6 - 1).$$%/(2e_1 - e_2 + 2e_5 - 1 + e_6)^2$$
The shift operators for $e_6\to e_6+2$ is obtained from above by the change $e_5\leftrightarrow e_6$.

\subsection{Equation $E_{3f}:\{e_2=-e_1-e_3-e_5+1,e_4=e_3-e_5+1\}$}
%G3 := subs({e2=-e1-e3-e5+1,e4=e3-e5+1}, dx^3*p3 + dx^2*p2 + dx*p1 + p0)
In $y$-coordinates: $\{y_3=-y_5,y_6=-y_2-y_1+2y_5-2y_4-y_7\}$. \par\noindent
\begin{thm}The equation $E_{3f}$ admits a shift operator for every shift $(e_1,e_6)\to(e_1\pm1,e_6\pm1)$.\end{thm}
%\begin{remark}We could not find a shift operator for other shift.%any shift $e_2\to e_2+n$.  \end{remark}
\noindent
{\bf Shift operator for $(e_1,e_6)\to(e_1+1,e_6+1)$}:
$$\begin{array}{ll}
P&=x(x-1)^2dx^2+(x-1)((e_5+e_6+1)x+e_1-1)dx+e_5e_6x-e_1/2-e_6/2+e_1^2/2\\
&+e_1e_3/2+e_5e_1/2+e_6e_1+e_6e_3/2-e_5e_6/2+e_6^2/2:\\
Q&=x(x-1)^2dx^2+(x-1)((e_5+e_6+2)x+e_1)dx+(e_5e_6+e_5)x+e_1^2/2+e_1e_3/2\\
&+e_5e_1/2+e_6e_1+e_6e_3/2-e_5e_6/2+e_6^2/2+e_1/2+e_6/2
  \end{array}$$
S-value:
$$(e_1+e_6)(e_1+2e_3+e_6)(e_1+e_3+e_5+e_6-1)(e_6+1-e_5+e_3+e_1).$$
{\bf Shift operator for $(e_1,e_6)\to(e_1-1,e_6+1)$}:
$$\begin{array}{ll}
P&=x(x-1)^2dx^2+(x-1)((e_5+e_6+1)x-e_1-e_3-e_5)dx+e_5e_6x\\
&-e_1/2+e_6/2+e_1^2/2+e_1e_3/2+e_5e_1/2-e_6e_1-e_6e_3/2-(3e_5e_6)/2+e_6^2/2,\\
Q&=x(x-1)^2dx^2+(x-1)((e_5+e_6+2)x-e_1-e_3-e_5+1)dx+(e_5e_6+e_5)x\\
&+e_1^2/2+e_1e_3/2+e_5e_1/2-e_6e_1-e_6e_3/2-(3e_5e_6)/2+e_6^2/2-(3e_1)/2\\
&-e_3-e_5+(3e_6)/2+1
  \end{array}$$
S-value:
$$(e_1-e_6)(e_1+2e_5-e_6-2)(e_1+e_3+e_5-e_6-1)(-e_6-1+e_5-e_3+e_1).$$

%\newpage
\section{Some generalities}\secttoc
In this section we extract some definitions and theorems from \cite{HOSY1} needed in this paper. Let $D:=\mathbb{C}(x)[\partial].$
\subsection{Shift operators and shift relations}
Let $E(e)$ be a Fuchsian differential equation of order 3, and $e=(e_1,e_2\dots)$ a system of local exponents. Assume $E(e)$ is irreducible for generic $e$.

\begin{dfn}\label{DefShift} Let ${\rm Sol}(E(e))$ be the solution space of $E(e)$. For a shift $sh_+:e\to e_+$:
  $$sh_+:(e_+)_i=e_i+n_i,\quad n_i\in\mathbb{Z},$$
  a non-zero operator $P\in D$ of order lower than $3$
  sending ${\rm Sol}(E(e))$ to ${\rm Sol}(E(e_+))$ is called
  a {\it shift operator} for the shift $sh_+$ 
  and is denoted by $P_{+}$. A shift operator for the shift 
 $ sh_-:(e_-)_i=e_i-n_i$
  is denoted by $P_{-}$.
\end{dfn}
Suppose a shift operator $P_{+}\in D$ for a shift $sh_+$ exists.
Since $E(e_+)\circ P_{+}$ is divisible from right by $E(e)$,
there is an operator $Q_{+}\in D$ satisfying the {\it shift relation}:
  $$(EPQE):\quad E(e_+)\circ P_{+}=Q_{+}\circ E(e).$$
Conversely, if there is a pair of non-zero operators $(P_{+},Q_{+})\in D^2$
of order smaller than $n$ satisfying this relation,
then $P_{+}$ is a shift operator for the shift $sh_+$.
We often call also the pair $(P_{+},Q_{+})$ the shift operator for $sh_+$.
\begin{prp}Notation as above, if $P_+$ exists, then $P_-$ exists, and vice versa. If $P_+$ exists, then it is unique up to multiplicative constant (independent of $(x,\partial)$). For every shift operator, we can assume that the  coefficients are polynomials of $e$ free of common factors.
\end{prp}

\begin{remark}
   When a differential {\it equation} in question is $Eu=0$, by multiplying a non-zero polynomial to the {\it operator} $E$, we can assume that $E$ has no poles. However, shift operators may have poles as functions of $x$.
  \end{remark}
\begin{prp}\label{expofQ} If an operator $E(e)$ with the adjoint symmetry
  $E(e)^*=\\E (adj(e))$    %%# added  'the' and linebreak
  admits a shift relation $E(\sigma(e))\circ P=Q\circ E(e)$, then 
  $$Q=(-)^\nu P(adj\circ \sigma(e))^*,\quad \nu={\rm order}(P).$$
\end{prp}
\subsection{S-values and reducibility conditions}\label{GenShiftS}
Let a shift $sh_+:e\to e_+$ and its inverse $sh_-:e\to e_-$ admit shift operators
$$P_+(e):{\rm Sol}(E(e))\to {\rm Sol}(E(e_+))$$ and
$$P_-(e):{\rm Sol}(E(e))\to {\rm Sol}(E(e_-)).$$
 Consider  compositions of shift operators:
$$P_+(e_-)\circ P_-(e):{\rm Sol}(E(e)\to {\rm Sol}(E(e_-)\to {\rm Sol}(E(e)),$$
and
$$P_-(e_+)\circ P_+(e):{\rm Sol}(E(e)\to {\rm Sol}(E(e_+)\to {\rm Sol}(E(e));$$
these are constants (times the identity).
\begin{dfn}\label{Svalues}These constants will be called the {\it S-values} for $sh_\mp$, and are denoted as 
$$Sv_{sh_-}=P_+(e_-)\circ P_-(e) {\rm\quad mod\quad} E(e)$$
and
$$Sv_{sh+}=P_-(e_+)\circ P+(e) {\rm\quad mod\quad} E(e).$$
\end{dfn}
\begin{prp}\label{twoSvalues} The two S-values are related as
  $$Sv_{sh_-}(e)=Sv_{sh_+}(e_-).$$
\end{prp}
\begin{prp}\label{Sred}   
If for some $e=\epsilon$, $Sv_{sh_+}(\epsilon)=0$ $($resp. $Sv_{sh_-}(\epsilon)=0)$, then $E(\epsilon)$ and $E(\epsilon_+)$ $($resp. $E(\epsilon_-))$ are reducible. If  $Sv_{sh_+}(\epsilon)\not=0$ $($resp. $Sv_{sh_-}(\epsilon)\not=0)$, then $P_{sh_+}$ $($resp. $P_{sh_-})$ gives an isomorphism: ${\rm Sol}(E(\epsilon))\to {\rm Sol}(E(\epsilon_+))$ $($resp. ${\rm Sol}(E(\epsilon))\to {\rm Sol}(E(\epsilon_-)))$.% as $\pi_1(\mathbb{C}-\{0,1\})$-modules.
\end{prp}
\subsection{Reducibility type and shift operators}
We discuss factorization of Fuchsian operators in $D=\mathbb{C}(x)[\partial]$.

\begin{dfn}When $E\in D$ is reducible and factorizes as
  $$E=F_1\circ \cdots\circ F_r,\quad F_j\in D,\quad 0<{\rm order}(F_j)=n_j,\ (j=1, \dots, r),$$
  we say $E$ is {\it reducible of type} $[n_1, \dots, n_r]$;
  we sometimes call $[n_1, \dots, n_r]$ the {\it type of factors}.
We often forget commas, for example, we write [12] in place of [1,2].
When only a set of factors matters,
we say $E$ is {\it reducible of type} $\{n_1, \dots, n_r\}$. 
\end{dfn}

Note that even if the equation $E$ has singularity only at $S=\{0, 1, \infty\}$,
the factors may have singularities out of $S$.

\begin{prp}\label{apparentsing} If $E$ has singularity only at $S$, then the singular points of $F_1$ and $F_r$ out of $S$ are apparent.
\end{prp}

\begin{remark}
The way of factorization is far from unique. When we discuss the singularity of the factors of a decomposition,
we usually choose the factors so that they have least number of singular points.
\end{remark}

\begin{prp}\label{FactorType}
  Suppose $E(e)$ and $E(e_\pm)$ are connected by shift relations. If $Sv_+(\epsilon)\not=0$ (resp. $Sv_-(\epsilon)\not=0$) for some $e=\epsilon$, then $E(\epsilon)$ and $E(\epsilon_+)$ (resp. $E(\epsilon_-)$ admit the factorization of the same type.
\end{prp}

\begin{thm}\label{onetwo_to_twoone} 
If an equation $E(e)$ admits a shift operator $P_+$ for a shift $sh_+:e\to e_+$, and if for some $e=\epsilon$, $E(\epsilon)$ is reducible of type $[1,2]$, then there are two cases:
 \begin{itemize}
\item $E(\epsilon_+)$ is reducible of type $[1,2]$; in this case, $E(sh_+^n(\epsilon))$ is also reducible of type $[1,2]$, and $Sv_{sh_+}(sh_+^n(\epsilon))\not=0$ for $n=1,2,\dots$, 
\item $E(\epsilon_+)$ is reducible of type $[2,1]$; in this case,  $E(sh_+^n(\epsilon))$ is reducible of type $[2,1]$, for $n=1,2,\dots$  %The factors of $E(\epsilon)$ and $E(\epsilon_+)$ have no apparent songularities, 
and $Sv_{sh_+}(\epsilon)=Sv_{sh_+}(\epsilon_+)=0$, and $Sv_{sh_+}(sh_+^n(\epsilon))\not=0$, for $n=2,3,\dots$.
\end{itemize}\end{thm} 
 \newpage
\bibliographystyle{amsalpha}

\bigskip

\noindent 
Yoshishige Haraoka

Josai University, Sakado 350-0295, Japan

haraoka@kumamoto-u.ac.jp

\medskip \noindent
Hiroyuki Ochiai

Institute of Mathematics for Industry, Kyushu University, Fukuoka 819-0395, Japan 

ochiai@imi.kyushu-u.ac.jp

\medskip \noindent 
Takeshi Sasaki

Kobe University, Kobe 657-8501, Japan 

yfd72128@nifty.com

\medskip
\noindent
Masaaki Yoshida

Kyushu University, Fukuoka 819-0395, Japan 

myoshida1948@jcom.home.ne.jp

\end{document}